\documentclass[12pt,leqno]{amsart}
\usepackage{amsmath,amsthm,amsfonts}
\usepackage{bm}
\usepackage{amssymb}
\usepackage{times}
\usepackage{mathrsfs}
\usepackage{enumerate}
\usepackage[dvipdfmx]{graphicx,color}
\usepackage{constants}

\theoremstyle{plain} 
\newtheorem{theorem}{Theorem}[section] 
\newtheorem{corollary}[theorem]{Corollary} 
\newtheorem{proposition}[theorem]{Proposition}
\newtheorem{lemma}[theorem]{Lemma} 
\newtheorem{prop}[theorem]{Proposition} 

\theoremstyle{definition} 
\newtheorem{definition}[theorem]{Definition}
\newtheorem{remark}[theorem]{Remark}

\newconstantfamily{eps}{symbol=\epsilon}
\newconstantfamily{c}{symbol=c}
\let\spt=\supp

\DeclareMathOperator{\dist}{dist}

\newcommand{\N}{\mathbb{N}}

\newcommand{\R}{\mathbb{R}}
\newcommand{\e}{\varepsilon}

\numberwithin{equation}{section}
\title{Convergence of the Allen-Cahn equation \\
with Neumann boundary conditions}
\author{Masashi Mizuno}
\author{Yoshihiro Tonegawa}

\thanks{M. Mizuno worked done during a visit to the Institut
Mittag-Leffler (Djursholm, Sweden). This work was supported by JSPS
KAKENHI Grant Numbers 21224001, 25800084, 25247008}

\address[Masashi Mizuno]{Department of Mathematics, College of Science
and Technology, Nihon University, Tokyo 101-8308 JAPAN}
\email{mizuno@math.cst.nihon-u.ac.jp}

\address[Yoshihiro Tonegawa]{Department of Mathematics, Hokkaido
University, Sapporo 060-0810 JAPAN}
\email{tonegawa@math.sci.hokudai.ac.jp}

\allowdisplaybreaks[1]
\begin{document}

\begin{abstract}
 We study a singular limit problem of the Allen-Cahn equation with Neumann
 boundary conditions and general initial data of uniformly
 bounded energy. We prove that the time-parametrized family of limit energy measures 
 is Brakke's mean curvature flow with a generalized right angle condition on the boundary. 
 \end{abstract}

\maketitle

\section{Introduction}

We consider the following Allen-Cahn equation:
\begin{equation}
 \label{eq:1.1}
  \left\{
   \begin{array}{ll}
    \partial_tu^\varepsilon
    =\Delta u^\varepsilon
    -\frac{W'(u^\varepsilon)}{\varepsilon^2},&
    t>0,\ x\in\Omega,
    \\
    \frac{\partial u^\varepsilon}{\partial\nu}\Big|_{\partial\Omega}
    =0,&
    t>0,\\
    u^\varepsilon(x,0)
    =u_0^\varepsilon(x),
    &x\in\Omega,
   \end{array}
  \right.
\end{equation}
where $\Omega\subset\R^n$ is a bounded domain with smooth boundary,
$\varepsilon>0$ is a small positive parameter, $\nu$ is the outer unit
normal vector field on $\partial\Omega$ and $W$ is a bi-stable potential with
two equal wells at $\pm 1$.  $W(u)=\frac14(1-u^2)^2$ is a typical example. The
equation \eqref{eq:1.1} is a gradient flow of
\begin{equation*}
 E^\varepsilon[u]
  :=\int_\Omega
  \left(
  \frac\varepsilon2|\nabla u|^2
  +\frac{W(u)}{\varepsilon}
  \right)
  \,dx
\end{equation*}
as one may check easily that $\frac{dE^{\varepsilon}}{dt}\leq 0$. Under the assumption 
that a given family $\{u^{\varepsilon}_0\}_{0<\varepsilon<1}$ satisfies
\begin{equation*}
\sup_{0<\varepsilon<1} E^{\varepsilon}[u^{\varepsilon}_0]<\infty,
\end{equation*}
it is interesting to study the limiting behavior of the solution $u^{\e}$ 
of \eqref{eq:1.1} as $\varepsilon\rightarrow 0$. Heuristically, one expects
that the finiteness assumption for $E^{\varepsilon}[u^{\varepsilon}(\cdot,t)]$ for very small $\varepsilon$ implies a `phase separation',
i.e., $\Omega$ is mostly divided into two regions where $u^{\varepsilon}(\cdot,t)$ is close to $1$ 
on one of them and to $-1$ on the other, 
with thin `transition layer' of order
$\varepsilon$ thickness separating these two regions. With this heuristic picture,
one may also expect that the following measures $\mu_t^{\e}$ defined by 
\begin{equation}
 d\mu_t^\varepsilon
  :=\left(
  \frac\varepsilon2|\nabla u^{\e}(x,t)|^2
  +\frac{W(u^{\e}(x,t))}{\varepsilon}
  \right)
  \,dx \label{defmu}
\end{equation}
behave more or less like surface measures of moving phase boundaries. 
It is thus interesting and natural to study
$\lim_{\varepsilon\rightarrow 0}\mu^{\varepsilon}_t$. 
By the well-known heuristic argument using 
the signed distance functions to the moving phase boundaries composed 
with the one-dimensional standing wave solution of $\varepsilon^2 u''=W'(u)$, 
one may also expect
that the motion of the phase boundaries is the mean curvature flow
(abbreviated hereafter as MCF). The 
rigorous proof of this in the most general setting, on the other hand, requires
extensive use of tools from geometric measure theory. 

The singular limit of \eqref{eq:1.1} without boundary is studied by many
researchers with different settings and assumptions.  The most relevant
among them to the present paper is Ilmanen's work \cite{MR1237490},
which showed that the limit measures of $\mu_t^\varepsilon$ are the MCF
in the sense of Brakke \cite{MR0485012} (where $\Omega=\R^n$). There was
a technical assumption in \cite{MR1237490} on the initial condition,
which was removed by Soner \cite{MR1674799}. The second author observed
that Ilmanen's work can be extended to bounded domains, and showed that
the limit measures have integer densities a.e$.$ modulo division by a
constant \cite{Tonegawa1}.  If the densities are equal to 1 a.e., it has
been proved recently that the support of the measures is smooth a.e$.$
as well \cite{MR0485012,arXiv:1111.0824,Tonegawa2}.  By these works,
interior behavior of the limit measures has been rigorously
characterized as Brakke's MCF.  There are numerous earlier and relevant
results on \eqref{eq:1.1} and we additionally mention
\cite{MR1101239,MR1153311,MR1055457,MR1177477,MR2383536,MR978829,MR2440879,Soner1,Soner3,arXiv:1307.6629}
which is by no means an exhaustive listing.

For the problem with Neumann boundary conditions, 
one may heuristically expect that the 
limit phase boundaries intersect $\partial\Omega$ with 90 degree angle. 
Katsourakis et al. 
\cite{MR1341031} basically proved this connecting 
the singular limit of \eqref{eq:1.1} to the unique
viscosity solutions of level set equations of the MCF with
right angle boundary conditions studied in
\cite{MR1235189,MR1287918}. 
The differences of the present paper from  
\cite{MR1341031} are explained as follows. While
one does not know in \cite{MR1341031} 
if the particular individual level set
obtained as a singular limit of \eqref{eq:1.1} satisfies MCF equation or
boundary conditions in some measure-theoretic sense, 
we show that the limit measure satisfies Brakke's inequality
with a generalized right angle condition. If we assume that
the limit measure has density 1 a.e$.$, then, it is smooth a.e$.$ in the
interior due to \cite{MR0485012,arXiv:1111.0824,Tonegawa2}. We also 
obtain a characterization for any finite energy initial data in $W^{1,2}(\Omega)$
and not necessarily for a carefully prepared initial data. 
Perhaps the most insightful aspect of the present paper is that
our study motivates a
measure-theoretic formulation of Brakke's MCF up to the 
boundary (see Section 2.4) for which one may further pursue the establishment of 
up to the boundary regularity theorem. 

More technically speaking, in this paper, we prove that (1) the limit measures $\mu_t$ 
have bounded first variation on $\overline\Omega$ for a.e$.$ $t\geq 0$,
(2) $\mu_t$ is $n-1$-rectifiable on $\overline\Omega$ and 
integral (modulo division by a constant) on $\Omega$ for a.e$.$ $t\geq 0$, (3) 
$\mu_t$ satisfies Brakke's inequality of MCF up to the boundary with a 
suitable modification
for the first variation on $\partial\Omega$. If we assume in addition that 
$\mu_t(\partial \Omega)=0$, then the right angle condition on the boundary is satisfied in
the sense that the first variation of $\mu_t$ on $\partial \Omega$ is perpendicular to $\partial \Omega$. We make an 
assumption that $\Omega$ is strictly convex, even though some generalization is possible
(see Section \ref{finalremark}). The proof uses various 
ideas developed through \cite{MR1237490,Tonegawa1,arXiv:1307.6629}. In those
paper, the Huisken/Ilmanen monotonicity formula played a central role and the 
situation is the same 
in this paper as well. We first prove up to the boundary monotonicity formula by 
a boundary reflection method, and this leads us to similar estimates as 
in the interior case. We need to be concerned with measures concentrated on $\partial \Omega$
as well as the limit of `boundary measures of phase boundary'. 
All those quantities are incorporated in the final formulation appearing in Theorem \ref{theorem2.6}. 

The paper is organized as follows. We explain notation and main results in Section 2. In Section 3 we obtain
up to the boundary monotonicity formula. The formula is not useful until we obtain
an $\e$-independent estimate on the 
so-called discrepancy in Section 4. Section 5 shows the existence of converging subsequence for
all time, and Section 6 shows the vanishing of the discrepancy which is the key to show the
main result. Combining all the ingredients, Section 7 finally proves the main results of the paper. 
\section{Preliminaries and main results}
\subsection{Basic notation}
Let ${\mathbb N}$ be the set of natural numbers and ${\mathbb R}^+:=\{x\geq 0\}$. 
For $0<r<\infty$ and $a\in {\mathbb R}^k$, define $B_r^k(a):=\{x\in {\mathbb R}^k\, :\,
|x-a|<r\}$. When $k=n$, we omit writing $k$ and we write $B_r:=B_r^n(0)$. 
The Lebesgue measure is denoted by ${\mathcal L}^n$ and the 
$k$-dimensional Hausdorff measure is denoted by ${\mathcal H}^k$. 
Let $\omega_n:={\mathcal L}^n (B_1)$. 

For any Radon measure $\mu$ on ${\mathbb R}^n$ and $\phi\in C_c({\mathbb R}^n)$
we often write $\mu(\phi)$ for $\int \phi\, d\mu$. We write ${\rm spt}\,\mu$ for the support 
of $\mu$. Thus $x\in {\rm spt}\, \mu$ if $\forall r>0$, $\mu(B_r(x))>0$. We use the standard
notation for the Sobolev spaces such as $W^{1,p}(\Omega)$ from \cite{Gilbarg}.

For $A,B\in {\rm Hom}({\mathbb R}^n;{\mathbb R}^n)$ which we identify with $n\times n$
matrices, we define
\begin{equation*}
A\cdot B:=\sum_{i,j} A_{ij}B_{ij}.
\end{equation*}
The identity of ${\rm Hom}({\mathbb R}^n;{\mathbb R}^n)$ is denoted by $I$. For 
$k\in {\mathbb N}$ with $k<n$, let ${\bf G}(n,k)$ be the space of $k$-dimensional
subspaces of ${\mathbb R}^n$. 
For $S\in {\bf G}(n,k)$, we identify $S$ with the corresponding orthogonal projection
of ${\mathbb R}^n$ onto $S$ and its matrix representation. 
For $a\in {\mathbb R}^n$, $a\otimes a\in {\rm Hom}
({\mathbb R}^n;{\mathbb R}^n)$ is the matrix with the entries $a_i a_j$ ($1\leq i,j\leq n$).
For any unit vector $a\in {\mathbb R}^n$, $I-a\otimes a\in {\bf G}(n,n-1)$. 
For $x,y\in {\mathbb R}^n$ and $t<s$, define
\begin{equation}
\rho_{(y,s)}(x,t):=\frac{1}{(4\pi (s-t))^{\frac{n-1}{2}}} e^{-\frac{|x-y|^2}{4(s-t)}}.
\label{defrho}
\end{equation}
\subsection{Varifold}
We recall some definitions related to varifolds and refer to \cite{MR0397520,MR756417}
for more details. In this paper, for a bounded open set $\Omega\subset{\mathbb R}^n$, 
we need to consider various objects on $\overline\Omega$ instead of $\Omega$. For 
this reason, let $X\subset {\mathbb R}^n$ be either open or compact in the following.  
Let $G_k(X):=X\times {\bf G}(n,k)$. A general
$k$-varifold in $X$ is a Radon measure on $G_k(X)$. 
We denote the set of all general $k$-varifold in $X$ by ${\bf V}_k (X)$. 
For $V\in {\bf V}_k (X)$, let $\|V\|$ be the weight measure of $V$,
namely, 
\begin{equation*}
\|V\|(\phi):=\int_{G_k(X)} \phi(x)\, dV(x,S),\ \ \forall \phi\in C_c(X).
\end{equation*}
We say $V\in {\bf V}_k(X)$ is rectifiable if there exist a ${\mathcal H}^k$ measurable
countably $k$-rectifiable set $M\subset X$ and a locally ${\mathcal H}^k$ integrable
function $\theta$ defined on $M$ such that 
\begin{equation}
V(\phi)=\int_M \phi(x, {\rm Tan}_x M)\theta(x)\, d{\mathcal H}^k
\label{recvar}
\end{equation}
for $\phi\in C_c(G_k(X))$. Here ${\rm Tan}_x M$ is the approximate tangent
space of $M$ at $x$ which exists ${\mathcal H}^k$ a.e$.$ on $M$. 
Rectifiable $k$-varifold is uniquely determined by its weight measure
through the formula \eqref{recvar}. For this reason, we naturally say a 
Radon measure $\mu$ on $X$ is rectifiable if there exists a rectifiable
varifold such that the weight measure is equal to $\mu$. If in addition that $\theta\in {\mathbb N}$
${\mathcal H}^k$ a.e$.$ on $M$, we say $V$ is integral. The set of all rectifiable (resp. integral)
$k$-varifolds in $X$ is denoted by ${\bf RV}_k(X)$ (resp. ${\bf IV}_k(X)$). If $\theta=1$ 
${\mathcal H}^k$ a.e$.$ on $M$, we say $V$ is a unit density $k$-varifold. 

For $V\in {\bf V}_k(X)$ let $\delta V$ be the first variation of $V$, namely, 
\begin{equation}
\delta V(g):=\int_{G_k(X)} \nabla g(x)\cdot S\, dV(x,S)
\label{deffirst}
\end{equation}
for $g\in C^1_c(X;{\mathbb R}^n)$. If the total variation $\|\delta V\|$ of $\delta V$ is
locally bounded (note in the case of $X=\overline\Omega$, this means $\|\delta V\|(\overline\Omega)
<\infty$), we may apply the Radon-Nikodym theorem to $\delta V$ with respect to $\|V\|$.  
Writing the singular part of $\|\delta V\|$ with respect to $\|V\|$ as $\|\delta V\|_{\rm sing}$, 
we have $\|V\|$ measurable $h(V,\cdot)$, $\|\delta V\|$ measurable $\nu_{\rm sing}$ with $|\nu_{\rm sing}|
=1$ $\|\delta V\|$ a.e$.$, 
and a Borel set $Z\subset X$ such that $\|V\|(Z)=0$ with,
\begin{equation*}
\delta V(g)=-\int_{X} h(V,\cdot)\cdot g \, d\|V\|+\int_{Z} \nu_{\rm sing}\cdot g\, d\|\delta V\|_{\rm sing}
\end{equation*}
for all $g\in C_c^1(X;{\mathbb R}^n)$. 
We say $h(V,\cdot)$ is the generalized mean curvature vector of $V$, 
$\nu_{\rm sing}$ is the (outer-pointing) generalized
co-normal of $V$ and $Z$ is the generalized boundary of $V$. 
\subsection{Setting of the problem}
Suppose that $n\geq 2$ and 
\begin{equation}
\Omega\subset {\mathbb R}^n\ \ \mbox{is a 
bounded, strictly convex domain with smooth boundary $\partial \Omega$.}
\label{ocond1}
\end{equation}
Here the strict convexity means that the principal curvatures of $\partial \Omega$ are all positive. 
 Suppose
that $W:{\mathbb R}\rightarrow{\mathbb R}$ is a $C^3$ function with 
$W(\pm 1)=0$, $W(u)\geq 0$ for all $u\in {\mathbb R}$, 
\begin{equation}
 \text{for some}\ -1<\gamma<1,\ W'<0\ \text{on}\ (\gamma,1)\ 
  \text{and}\ W'>0\ \text{on}\ (-1,\gamma),
  \label{wcond1}
\end{equation}
\begin{equation}
 \text{for some}\ 0<\alpha<1\ \text{and}\ \kappa>0,\  
  W''(u)\geq\kappa\ \text{for all}\ \alpha\leq|u|\leq1.
  \label{wcond2}
\end{equation}
A typical example of such $W$ is $(1-u^2)^2/4$, for which we may
set $\gamma=0$, $\alpha=\sqrt{2/3}$ and $\kappa=1$. For a given
sequence of positive numbers $\{\e_i\}_{i=1}^{\infty}$ with $\lim_{i\rightarrow\infty}
\e_i=0$, suppose that $u_0^{\e_i}\in W^{1,2}(\Omega)$ satisfies
\begin{equation}
\|u_0^{\e_i}\|_{L^{\infty}(\Omega)}\leq 1
\label{absu}
\end{equation}
and
\begin{equation}
\sup_{i} E^{\e_i}[u_0^{\e_i}]\leq \Cl[c]{c-1}.
\label{engu}
\end{equation}
The condition \eqref{absu} may be dropped if we assume a suitable growth rate
upper bound on $W$ which is suitable for the existence of solution for \eqref{eq:1.1}. A typical example 
of sequence of $u^{\e_i}_0$ may be given as in \cite{Modica}. We include the detail for the convenience
of the reader. 
 Let $U\subset {\mathbb R}^n$
be any domain with $C^1$ boundary $M=\partial U$, and let $\Phi$ be a solution of ODE $\Phi''=W'(\Phi)$
with $\Phi(\pm \infty)=\pm 1$ and $\Phi(0)=0$. Note that such a solution exists uniquely, 
and $\Phi$ also satisfies $\Phi'=\sqrt{2W(\Phi)}$. 
Let $d$ be the signed distance function to $M$ so that it is 
positive inside of $U$. Define $u_0^{\e_i}(x):=\Phi(d(x)/\e_i)$ for $x\in \Omega$.
Then one can check that, using $\Phi'=\sqrt{2W(\Phi)}$ and $|\nabla d|=1$ a.e$.$,
\begin{equation}
E^{\e_i}[u_0^{\e_i}]=\int_{\Omega}\e_i^{-1} (\Phi')^2\, dx=\int_{\Omega} \e_i^{-1}\Phi' \sqrt{2W(\Phi)}
|\nabla d|\, dx.
\label{eq:1.2}
\end{equation}
By the co-area formula, then, 
\begin{equation}
E^{\e_i}[u_0^{\e_i}]=\int_{-\infty}^{\infty} \int_{\Omega\cap \{d=\e_i s\}} \Phi'(s)\sqrt{2W(\Phi(s))}\,d{\mathcal H^{n-1}}
ds.
\label{eq:1.3}
\end{equation}
If $M$ is transverse to $\partial \Omega$, ${\mathcal H}^{n-1}(\Omega\cap \{d=\e_i s\})
\approx {\mathcal H}^{n-1}(M\cap \Omega)$ for small $\e_i$ and \eqref{eq:1.3} shows
\begin{equation}
\lim_{i\rightarrow\infty}E^{\e_i}[u_0^{\e_i}]= \sigma {\mathcal H}^{n-1}(\Omega\cap M), \ \ \
\sigma:=\int_{-1}^1\sqrt{2W(u)}\, du.
\label{eq:1.4}
\end{equation}
Thus in this case, we may take $\Cr{c-1}=\sigma {\mathcal H}^{n-1}(M\cap \Omega)+1$, for example. 

We next solve the problem \eqref{eq:1.1} with $\e_i$ and $u_0^{\e_i}$ satisfying \eqref{absu} and \eqref{engu}.
By the standard parabolic existence and regularity theory, for each $i$, 
there exists a unique solution $u^{\e_i}$ with
\begin{equation}
u^{\e_i}\in L^2_{loc}([0,\infty);W^{2,2}(\Omega))\cap C^{\infty}(\overline{\Omega}\times(0,\infty)),
\ \ \ \partial_t u^{\e_i}\in L^2([0,\infty);L^2(\Omega)).
\label{ubel}
\end{equation}
By the maximum principle and \eqref{absu},
\begin{equation}
\sup_{x\in \overline\Omega,\ t>0}|u^{\e_i}(x,t)|\leq  1,
\label{absu2}
\end{equation}
and due to the gradient structure and \eqref{engu}, we also have
\begin{equation}
E^{\e_i}[u^{\e_i}(\cdot,T)]+\int_0^{T}\int_{\Omega}
\e_i\big(\Delta u^{\e_i}-\frac{W'}{\e_i^2}\big)^2\, dxdt= E^{\e_i}[u^{\e_i}(\cdot,0)]\leq  \Cr{c-1}
\label{engu2}
\end{equation}
for any $T>0$. 
Thus, for each $i$ through \eqref{defmu}, we have a family $\{\mu_t^{\e_i}\}_{t\in [0,\infty)}$
of uniformly bounded Radon measures. 
\subsection{Main results}\label{main}
The following sequence of theorems and definitions constitutes the main results of the present paper. 
\begin{theorem}
\label{theorem2.1}
Under the assumptions \eqref{ocond1}-\eqref{engu}, let $u^{\e_i}$ be
the solution of \eqref{eq:1.1}. Define $\mu_t^{\e_i}$ as in \eqref{defmu}. 
Then there exists a subsequence (denoted by
the same index) and a family of Radon measures $\{\mu_t\}_{t\geq 0}$
on $\overline\Omega$ such that
for all $t\geq 0$, $\mu_t^{\e_i}\rightharpoonup\mu_t$ as $i\rightarrow\infty$ 
on $\overline\Omega$. Moreover,
for a.e$.$ $t\geq 0$, $\mu_t$ is rectifiable on $\overline \Omega$.
\end{theorem}
Due to Theorem \ref{theorem2.1}, we may define rectifiable varifolds as follows.
\begin{definition}
For a.e$.$ $t\geq 0$, let $V_t\in {\bf RV}_{n-1}(\overline\Omega)$ be the unique rectifiable varifold such that
$\|V_t\|=\mu_t$ on $\overline\Omega$. For any $t$ such that $\mu_t$ is not
rectifiable, define $V_t\in {\bf V}_{n-1}(\overline\Omega)$ to be an arbitrary varifold with $\|V_t\|=\mu_t$ 
(for example $V_t(\phi):=\int_{\overline\Omega}\phi(\cdot,{\mathbb R}^{n-1}\times\{0\})\, d\mu_t$
for $\phi\in C(G_{n-1}(\overline\Omega))$). 
\end{definition}

\begin{theorem}
\label{theorem2.3} 
Let $V_t$ be defined as above. Then the following property holds. 
\begin{itemize}
\item[(1)] For a.e$.$ $t\geq 0$, $\sigma^{-1}V_t\lfloor_{\Omega}\in {\bf IV}_{n-1}(\Omega)$. 
\item[(2)] For a.e$.$ $t\geq 0$, $\|\delta V_t\|(\overline\Omega)<\infty$ and $\int_{0}^T \|\delta V_t\|
(\overline\Omega)\, dt<\infty$ for all $T>0$. 
\end{itemize}
\end{theorem}
We next define the tangential component of the first variation $\delta V_t$ on $\partial\Omega$. 
\begin{definition}
For a.e$.$ $t\geq 0$ such that $\|\delta V_t\|(\overline\Omega)<\infty$, define
\begin{equation}
\delta V_t\lfloor_{\partial\Omega}^{\top} (g):=\delta V_t\lfloor_{\partial\Omega}(g-(g\cdot\nu)\nu)
\ \ \ \mbox{for } g\in C(\partial \Omega;{\mathbb R}^n)
\label{deftan}
\end{equation}
where $\nu$ is the unit outward-pointing normal vector field on $\partial\Omega$. 
\end{definition}
We have the following absolute continuity result. 
\begin{theorem}
\label{theorem2.5}
For a.e$.$ $t\geq 0$, we have $\|\delta V_t\lfloor_{\partial\Omega}^{\top}+\delta V_t\lfloor_{\Omega}\|
\ll \|V_t\|$, and there exists $h_b=h_b(t)\in L^2(\|V_t\|)$ such that 
\begin{equation}
\delta V_t\lfloor_{\partial\Omega}^{\top}+\delta V_t\lfloor_{\Omega}=-h_b(t)\|V_t\|.
\label{hel2b}
\end{equation}
Moreover,
\begin{equation}
\label{hel2}
\int_0^{\infty} \int_{\overline\Omega} |h_b|^2\, d\|V_t\|dt\leq \Cr{c-1}.
\end{equation}
\end{theorem}
Note that $h_b=h(V_t,\cdot)$ in $\Omega$. Finally, using the above quantities, we have
\begin{theorem}
\label{theorem2.6}
For $\phi \in C^1(\overline\Omega\times[0,\infty)\,;\,\R^+)$ 
with $\nabla \phi(\cdot,t)\cdot\nu=0$ on $\partial\Omega$ and for any $0\leq t_1<t_2<\infty$,
we have
\begin{equation}
\label{mainineq}
\int_{\overline\Omega} \phi(\cdot,t)\, d\|V_t\|\Big|_{t=t_1}^{t_2}\leq
\int_{t_1}^{t_2}\int_{\overline\Omega} \big(-\phi |h_b|^2+\nabla\phi\cdot h_b
+\partial_t\phi\big)\, d\|V_t\|dt.
\end{equation}
\end{theorem}
If $\phi(\cdot,t)$ has a compact support in $\Omega$, \eqref{mainineq} is Brakke's inequality
\cite{MR0485012} in an integral form. If we have a situation that $\|V_t\|(\partial \Omega)=0$,
then Theorem \ref{theorem2.5} shows $\delta V_{t}\lfloor_{\partial\Omega}^{\top}=0$
and $\delta V_t \lfloor_{\partial\Omega}$ is singular with respect to $\|V_t\|$.
It is parallel to $\nu$ for $\|\delta V_t\|$ a.e$.$ which would, if ${\rm spt}\,\|V_t\|$ is smooth
up to the boundary, correspond to 90 degree angle of intersection. 
The reader is referred to Section \ref{finalremark} for further
remarks on the above formulation. 
\section{Boundary monotonicity formula}
\label{sec:3}
The first task of our problem is to establish some up-to the boundary
monotonicity formula of Huisken/Ilmanen type. 
Define $\Cl[c]{c-2}$ by
\begin{equation*}
  \Cr{c-2}
 :=(\|\mbox{principal curvatures of $\partial \Omega$}\|_{L^{\infty}(\partial \Omega)})^{-1}.
\end{equation*}
Since $\partial\Omega$ is assumed to be smooth and compact, $0<\Cr{c-2}<\infty$. 
For $r\leq \Cr{c-2}$, let us denote by $N_r$ the interior tubular neighborhood of $\partial\Omega$, namely
\begin{equation*}
  N_r:=\{x-\lambda\nu(x):x\in\partial\Omega,\ 0\leq \lambda<r\},
\end{equation*}
where $\nu$ is the unit outer-pointing normal vector field to $\partial \Omega$.
For $x\in N_{\Cr{c-2}}$, there exists a unique point
$\zeta(x)\in\partial\Omega$ such that
$\dist(x,\partial\Omega)=|x-\zeta(x)|$. We define the reflection point
$\tilde{x}$ of $x$ with respect to $\partial\Omega$ as
$\tilde{x}:=2\zeta(x)-x$ (see Figure \ref{Fig:3.1}).  We also fix a
radially symmetric function $\eta\in C^{\infty}({\mathbb R}^n)$ such
that
\begin{equation}
0\leq \eta\leq 1,\ \ \frac{\partial\eta}{\partial r}\leq 0, \ \  {\rm
 spt}\, \eta\subset B_{\Cr{c-2}/2},\ \ \eta=1 \mbox{ on }B_{\Cr{c-2}/4}.
\label{etadef}
\end{equation}
For $s>t>0$ and $x,y\in N_{\Cr{c-2}}$, we define the $(n-1)$-dimensional
reflected backward heat kernel denoted by 
$\tilde{\rho}_{(y,s)}(x,t)$ as
\begin{equation}
  \tilde\rho_{(y,s)}(x,t)
  :=\rho_{(y,s)}({\tilde x},t),
 \label{refheat}
\end{equation}
where $\rho_{(y,s)}$ is defined as in \eqref{defrho}. For 
$x,y\in N_{\Cr{c-2}}$, we define truncated versions of $\rho_{(y,s)}$ and 
$\tilde\rho_{(y,s)}$ as
\begin{equation}
\rho_1=\rho_1(x,t)=\eta(x-y)\rho_{(y,s)}(x,t) \ \ \mbox{ and } \ \ \
\rho_2=\rho_2(x,t)=\eta(\tilde{x}-y)\tilde\rho_{(y,s)}(x,t). 
\label{refheat2}
\end{equation}
For 
$x\in N_{\Cr{c-2}}\setminus N_{\Cr{c-2}/2}$ and $y\in N_{\Cr{c-2}/2}$, 
we have $|{\tilde x}-y|>\Cr{c-2}/2$. Thus we may smoothly define $\rho_2=0$
for $x\in \Omega\setminus N_{\Cr{c-2}/2}$ and $y\in N_{\Cr{c-2}/2}$. 
We also define a (signed) measure
\begin{equation}
 d\xi_t^\varepsilon
 =\left(\frac\varepsilon2|\nabla
 u^\varepsilon|^2-\frac{W(u^\varepsilon)}{\varepsilon}\right)\,dx
 \label{defxi}
\end{equation}
where the right-hand side is evaluated at time $t$. 

\begin{figure}
 \begin{center}
  \includegraphics[height=100pt]{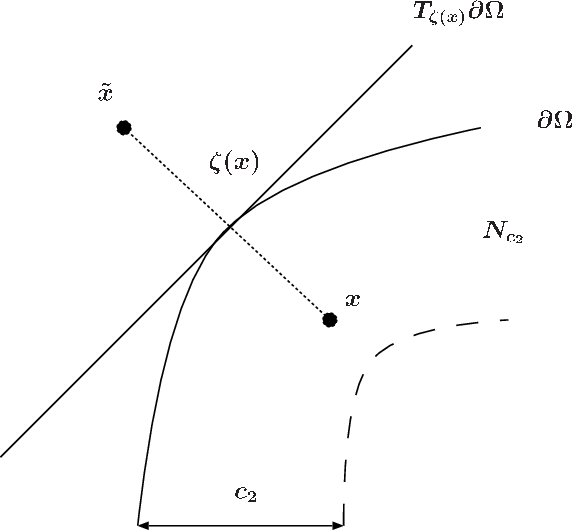}
 \end{center}
 \caption{The interior tubular neighbourhood and the reflection point
 $\tilde{x}$}
  \label{Fig:3.1}
\end{figure}

\begin{proposition}
 [Boundary monotonicity formula]
 \label{Prop:3.1}
 There exist $0<\Cl[c]{c-3},\Cl[c]{c-4}<\infty$ depending only on $n$,
 $\Cr{c-1}$ and $\Cr{c-2}$ such that
 \begin{equation}
  \frac{d}{dt}
  \big(e^{\Cr{c-3}(s-t)^{\frac14}}
  \int_{\Omega}  (
 \rho_1+\rho_2)
  \,d\mu^\varepsilon_t(x) \big)\leq 
 e^{\Cr{c-3}(s-t)^{\frac14}}\big(\Cr{c-4}
+
  \int_{\Omega}
  \frac{\rho_1+\rho_2}
  {2(s-t)}\,d\xi_t^\varepsilon(x)\big)
  \label{eq:3.7}
  \end{equation}
 for $s>t>0$ and $y\in N_{\Cr{c-2}/2}$. For $s>t>0$ and $y\in \Omega\setminus N_{\Cr{c-2}/2}$,
 we have
 \begin{equation}
 \frac{d}{dt}\int_{\Omega} \rho_1\, d\mu_t^{\e}(x)\leq 
 \Cr{c-4}+\int_{\Omega}\frac{\rho_1}{2(s-t)}\, d\xi_t^{\e}(x).
 \label{eq:3.75}
 \end{equation}
\end{proposition}

Above monotonicity formula is an analogue of Ilmanen's monotonicity
formula in ${\mathbb R}^n$ without boundary \cite{MR1237490}, which is the `Allen-Cahn equation version'
of Huisken's monotonicity formula for the MCF
\cite{MR1030675}.
Huisken's monotonicity formula for the MCF with the
90 degree angle boundary condition is derived by Stahl~\cite{MR1393271,MR1402731}, Buckland~\cite{MR2180601} and
Koeller~\cite{MR2886118}.
For stationary case of \eqref{eq:1.1}, the second author derived a
boundary monotonicity formula using the reflection argument \cite{MR1970021}, and just as in the
case of MCF, it is a `diffuse interface version' of a boundary monotonicity formula
for stationary varifold derived by 
Gr\"uter-Jost \cite{MR0863638}.

To derive Huisken's as well as Ilmanen's monotonicity formula, 
\begin{equation}
 \label{eq:3.3}
 \frac{(a\cdot\nabla\rho)^2}{\rho}+\left((I-a\otimes a)\cdot
 \nabla^2\rho\right)
 +\partial_t\rho=0
\end{equation}
is the crucial identity.  Here, $\rho=\rho_{(y,s)}(x,t)$ and $a=(a_j)$
is any unit vector. Before proving the boundary monotonicity formula, we
derive a similar identity for the reflected backward heat kernel
$\tilde{\rho}_{(y,s)}$.

\begin{lemma}
 \label{lem:3.1}
 For $a$ with $|a|=1$ and $\tilde\rho=\tilde\rho_{(y,s)}(x,t)$, we have
 \begin{equation}
   \label{eq:3.4}
    \frac{(a\cdot\nabla\tilde{\rho})^2}{\tilde{\rho}}
   +((I-a\otimes a)\cdot\nabla^2\tilde{\rho})
   +\partial_t\tilde{\rho}
   \leq C\left(
	  \frac{|\tilde{x}-y|}{s-t}
	  +\frac{|\tilde{x}-y|^3}{(s-t)^2} 
	 \right)
   \tilde{\rho}
 \end{equation}
for $0<t<s$ and $x,y\in N_{\Cr{c-2}}$ where
$C>0$ is some constant.
\end{lemma}

To prove Lemma \ref{lem:3.1}, we use the following lemma.
\begin{lemma}
 [cf.\cite{MR0397520,MR0863638}]
 Let 
 \[
 Q(x):=\nabla\zeta(x)-(I-\nu\otimes\nu),
 \]
 where $\nu$ is the unit normal vector at
 $\zeta(x)\in\partial\Omega$. Then
 \begin{enumerate}
  \item $Q(x)$ is symmetric;
  \item $Q(x)\nu=0$ for all $x\in N_{c_2/2}$;
  \item $|Q(x)|\leq 2|x-\zeta(x)|$ for all $x\in N_{c_2/2}$;
  \item If $\partial\Omega\in C^3$, then $|\nabla Q|$ is bounded.
 \end{enumerate}
\end{lemma}

For $x,y\in N_{c_2/2}$ by convexity
\[
 |x-\zeta(x)|
 =\frac12|x-\tilde{x}|
 \leq\frac12(|x-y|+|y-\tilde{x}|)
 \leq|\tilde{x}-y|,
\]
thus $|Q(x)|\leq 2|\tilde{x}-y|$.

\begin{proof}%
 [Proof of Lemma \ref{lem:3.1}]
 Since
 $\nabla \zeta(x)=I-\nu\otimes \nu+Q(x)$ and $\tilde x=2\zeta(x)-x$,
 we have 
 \begin{equation}
 \label{eq:3.4s}
  \begin{split}
   \nabla|\tilde{x}-y|^2&=2(I-2\nu\otimes\nu+2Q(x))(\tilde{x}-y), \\
   |\nabla|\tilde{x}-y|^2|&=4|(I-2\nu\otimes\nu+2Q(x))(\tilde{x}-y)|^2 \\
   &\leq4|\tilde{x}-y|^2+32|\tilde{x}-y|^3 \\
   \nabla^2_{ij} 
   |\tilde{x}-y|^2&=2\delta_{ij}-4\sum_{k}(\partial_{x_j}(\nu_i
   \nu_k)-\partial_{x_j}q_{ik})(\tilde x_k-y_k) \\
   &\quad+8q_{ij}+8\sum_{k=1}^n(q_{ik}q_{jk}-\nu_i\nu_kq_{jk}-\nu_j\nu_kq_{ik}).
  \end{split}
 \end{equation}
 where $Q(x)=(q_{ij})$. By direct calculation and \eqref{eq:3.4s}, we have
 \begin{equation}
 \label{eq:3.4t}
  \begin{split}
   \partial_t\tilde\rho&=\left(\frac{n-1}{2(s-t)}-\frac{|\tilde{x}-y|^2}{4(s-t)^2}\right)\tilde\rho, \ \ 
   \nabla\tilde\rho=-\frac{\nabla|\tilde x-y|^2}{{4(s-t)}}\tilde\rho, \\
   \nabla^2\tilde\rho
   &=\left(\frac{\nabla|\tilde x-y|^2\otimes \nabla|\tilde x-y|^2}{16(s-t)^2}
   -\frac{D^2|\tilde{x}-y|^2}{4(s-t)}
   \right)\tilde\rho
  \end{split}
 \end{equation}
 Using \eqref{eq:3.4t},
 we obtain \eqref{eq:3.4}.
\end{proof}

\begin{proof}
 [Proof of Proposition \ref{Prop:3.1}]
  By integration by part and using \eqref{eq:1.1} and denoting 
$f^{\varepsilon}:=-\varepsilon\Delta u^{\varepsilon}
 +\frac{W'(u^{\varepsilon})}{\varepsilon}$, we may obtain for each $i=1,2$
 \begin{equation}
  \label{eq:3.2}
   \begin{split}
    \frac{d}{dt}\int_\Omega\rho_{i}\,d\mu_t^\varepsilon
    &=-\frac1\varepsilon\int_\Omega
    (f^{\varepsilon})^2\rho_{i}\,dx+\int_\Omega
     f^{\varepsilon}
    \nabla\rho_{i}\cdot\nabla u^\varepsilon\,dx+\int_\Omega\partial_t \rho_{i}\,d\mu_t^\varepsilon \\
    &=\int_\Omega
    -\frac1\varepsilon\left(f^{\varepsilon}
    -\frac{\varepsilon\nabla u^\varepsilon\cdot\nabla\rho_i}{\rho_i}
    \right)^2\rho_i+\varepsilon
    \frac{(\nabla u^\varepsilon\cdot\nabla\rho_i)^2}{\rho_i} \,dx \\
    &\ \ \ -\int_\Omega     f^{\varepsilon}
    \nabla\rho_{i}\cdot\nabla u^\varepsilon\,dx +\int_\Omega\partial_t\rho_{i}\,d\mu_t^\varepsilon.
   \end{split}
 \end{equation}
By integration by parts again, we have
\begin{equation}
 \begin{split}
  \quad-\int_\Omega
  f^{\varepsilon}
  \nabla\rho_{i}\cdot\nabla u^\varepsilon
  \,dx 
  & =\int_\Omega
  \Delta \rho_{i}
  \,d\mu_t^\varepsilon
  -\int_\Omega
  \varepsilon (\nabla u^{\varepsilon}\otimes \nabla u^{\varepsilon}\cdot\nabla^2\rho_{i})
  \,dx \\
  &\quad-\int_{\partial\Omega}\nabla\rho_{i}\cdot\nu
  \left(\frac{\varepsilon}{2}|\nabla u^\varepsilon|^2
  +\frac{W(u^\varepsilon)}{\varepsilon}
  \right)\,d{\mathcal H}^{n-1}.
 \end{split} 
 \label{eq:3.2n}
\end{equation}
 In the following,  denote $a^\varepsilon
 =\frac{\nabla u^\varepsilon}{|\nabla u^\varepsilon|}$.
 For $x\in \partial\Omega$, $x=\tilde x$ and one can check that
 $\nabla |\tilde x-y|^2\cdot \nu+\nabla |x-y|^2\cdot \nu=0$, which implies
 $\nabla(\rho_1+\rho_2)\cdot\nu\big|_{\partial\Omega}\equiv0$.
 Therefore we may obtain (using also $\mu_t^{\varepsilon}=\varepsilon |\nabla u^{\varepsilon}|^2
 -\xi_t^{\varepsilon}$)
\begin{equation*}
 \begin{split}
  \frac{d}{dt}\int_\Omega\rho_{1}+\rho_2\,d\mu_t^\varepsilon
  &\leq
  \sum_{i=1,2}\int_\Omega\left(
  \frac{(a^\varepsilon\cdot\nabla\rho_i)^2}{\rho_i}
  +\left((I-a^\varepsilon\otimes a^\varepsilon)\cdot\nabla ^2\rho_i\right)
  +\partial_t\rho_i
  \right)\varepsilon |\nabla u^{\varepsilon}|^2\, dx
   \\
   &\quad -  \sum_{i=1,2}\int_{\Omega} (\partial_t \rho_i+\Delta \rho_i)\, d\xi_t^{\varepsilon}.
  \end{split}
\end{equation*} 
  
Note that $\rho_i$ is bounded uniformly on $\{ |\nabla \eta|\neq 0\}$.
 Using this fact, \eqref{eq:3.3} and \eqref{eq:3.4} we may
 obtain
 \begin{equation}
  \frac{(a^\varepsilon\cdot\nabla\rho_1)^2}{\rho_1}
   +\left((I-a^\varepsilon\otimes a^\varepsilon)\cdot\nabla^2\rho_1\right)
   +\partial_t\rho_1 
   \leq \Cr{c-4}
   \label{eq:3.add1}
 \end{equation} 
 and
 \begin{equation}
  \begin{split}
   &\quad\frac{(a^\varepsilon\cdot\nabla
   \rho_2)^2}{\rho_2}
   +\left((I-a^\varepsilon\otimes a^\varepsilon)\cdot\nabla^2\rho_2\right)
   +\partial_t\rho_2 \\
   &\leq\sum_{i,j,k=1}^n\left(
   \frac{(\delta_{ij}-n_in_j)\partial_{x_j}(\nu_i\nu_k)(\tilde{x}_k-y_k)}%
   {2(s-t)}\right)
   \rho_2
   +\Cr{c-4}\leq
   \Cr{c-3}\left(\frac{|\tilde{x}-y|}{s-t}
   +\frac{|\tilde{x}-y|^3}{(s-t)^2}\right)
   \rho_2
   +\Cr{c-4}
  \end{split} 
  \label{eq:3.add2}
\end{equation} 
 for some constant $\Cr{c-3},\, \Cr{c-4}>0$ depending only on $n$ and $\Cr{c-2}$.
 In the following $\Cr{c-3}$ and $\Cr{c-4}$ may be different constants which depend
 only on $n,\Cr{c-1},\Cr{c-2}$.
 To compute the integration of \eqref{eq:3.add2}, we decompose the integration as
 \begin{equation*}
  \begin{split}
   \int_\Omega\frac{\Cr{c-3}|\tilde{x}-y|}{s-t}\rho_2 \varepsilon |\nabla u^{\varepsilon}|^2
   \,dx
   &\leq \Cr{c-3}\int_{\Omega\cap\{|\tilde{x}-y|\leq(s-t)^\frac14\}}\frac{|\tilde{x}-y|}{s-t}\rho_2
   \,d\mu_t^\varepsilon \\
   &\quad+\Cr{c-3}\int_{\Omega\cap\{|\tilde{x}-y|\geq(s-t)^\frac14\}}\frac{|\tilde{x}-y|}{s-t}\rho_2
   \,d\mu_t^\varepsilon
   =:I_1+I_2, \\
   \int_\Omega\frac{|\tilde{x}-y|^3}{(s-t)^2}\rho_2 \varepsilon |\nabla u^{\varepsilon}|^2
   \,dx
   &\leq \int_{\Omega\cap\{|\tilde{x}-y|\leq(s-t)^\frac{5}{12}\}}\frac{|\tilde{x}-y|}{s-t}\rho_2
   \,d\mu_t^\varepsilon \\
   &\quad+\int_{\Omega\cap\{|\tilde{x}-y|\geq(s-t)^\frac{5}{12}\}}\frac{|\tilde{x}-y|}{s-t}\rho_2
   \,d\mu_t^\varepsilon=:I_3+I_4.
  \end{split}
 \end{equation*}
 $I_1$ is estimated by
 \begin{equation}
  \label{eq:3.5}
  I_1\leq \Cr{c-3}(s-t)^{-\frac34}
  \int_{\Omega\cap\{|\tilde{x}-y|<(s-t)^\frac14\}}
  \rho_2\,d\mu_t^\varepsilon
  \leq \Cr{c-3}(s-t)^{-\frac34}
  \int_{\Omega}
  \rho_2\,d\mu_t^\varepsilon. 
 \end{equation}
 We may estimate $I_2$ as
 \begin{equation}
  \label{eq:3.6}
  I_2\leq \frac{\Cr{c-3}}{(s-t)^{1+\frac{n-1}2}}e^{\frac{-1}{4\sqrt{s-t}}}
   \int_{\Omega\cap {\rm spt}\,\rho_2}|\tilde{x}-y|
   \,d\mu_t^\varepsilon
   \leq \Cr{c-4}
 \end{equation}
 with an appropriately chosen $\Cr{c-4}$. $I_3$ and $I_4$ are
 estimated as a similar manner. 
 
 Therefore from \eqref{eq:3.5} and \eqref{eq:3.6}  we obtain
\begin{equation*}
\begin{split}
 &\quad \int_\Omega
 \left(\frac{(a^\varepsilon\cdot\nabla
 \rho_2)^2}{\rho_2}
 +\left((I-a^\varepsilon\otimes a^\varepsilon)\cdot\nabla^2\rho_2\right)
 +\partial_t\rho_2\right)
 \varepsilon|\nabla u^{\varepsilon}|^2\,dx \\
 &\leq \Cr{c-3}(s-t)^{-\frac34}
 \int_{\Omega}
 \rho_2\,d\mu_t^\varepsilon 
 +\Cr{c-4}.
\end{split}
\end{equation*}
 Almost a similar calculation shows that
 \begin{equation*}
  \begin{split}
   &\quad - \int_\Omega (\partial_t \rho_1+\Delta\rho_1)\, d\xi_t^{\varepsilon}
 \leq \int_\Omega\frac{\rho_1}{2(s-t)}\,d\xi_t^\varepsilon
   +\Cr{c-4}
  \end{split} 
\end{equation*}  
 and
 \begin{equation*}
  \begin{split}
   &\quad - \int_\Omega
   (\partial_t\rho_2+\Delta\rho_2)\, d\xi_t^{\varepsilon}
   \leq 
   \int_\Omega\frac{\rho_2}{2(s-t)}\,d\xi_t^\varepsilon
   +\Cr{c-3}(s-t)^{-\frac34}
   \int_{\Omega}
   \rho_2\,d\mu_t^\varepsilon 
   +\Cr{c-4}.
  \end{split} 
\end{equation*}
 Therefore, we have

 \begin{equation*}
  \frac{d}{dt}\int_{\Omega}(\rho_1+\rho_2)\,d\mu_t^\varepsilon 
   \leq
   \frac{\Cr{c-3}}{(s-t)^{\frac34}}\int_{\Omega}(\rho_1+\rho_2)\,d\mu_t^\varepsilon
   +\Cr{c-4}
   +\int_\Omega\frac{\rho_1+\rho_2}{2(s-t)}\,d\xi_t^\varepsilon.
 \end{equation*}
 This leads to \eqref{eq:3.7}. The inequality
 \eqref{eq:3.75} can be obtained by observing that ${\rm spt}\,\rho_1\subset \Omega$
 for $y\in \Omega\setminus N_{\Cr{c-2}/2}$ and by following the same but simpler 
 computation with $\rho_2\equiv 0$. 
\end{proof}
We use the following estimate later. 
\begin{proposition}
\label{Prop:3.3}
There exists a constant $\Cl[c]{c-5}$ depending only on $n,\Cr{c-1}$ and $\Cr{c-2}$ with
\begin{equation}
\int_{t}^{t+1}\int_{\partial \Omega} \Big(\frac{\varepsilon}{2}|\nabla u^{\varepsilon}|^2+\frac{W(u^{\varepsilon})}{\varepsilon}\Big)
\, d{\mathcal H}^{n-1}dt\leq \Cr{c-5}
\label{surf}
\end{equation}
for any $t\geq 0$. 
\end{proposition}
\begin{proof}
[Proof of Proposition \ref{Prop:3.3}]
Let $\phi\in C^2(\overline\Omega)$ be a non-negative function so that, near $\partial \Omega$, $\phi(x)={\rm dist}\,(x,\partial \Omega)$,
and smoothly becomes a constant function on $\Omega\setminus N_{\Cr{c-2}/2}$. We may construct such function so that
$\|\phi\|_{C^2(\overline\Omega)}$ is bounded only in terms of $\Cr{c-2}$ and $n$.
Below, we use $\nabla\phi\cdot \nu=-1$ on $\partial \Omega$. 
We then compute as in the first line of \eqref{eq:3.2} and \eqref{eq:3.2n} with $\rho_i$ there replaced by $\phi$. By \eqref{engu2}
and dropping a negative term on the right-hand side,
we obtain
\begin{equation*}
\frac{d}{dt}\int_{\Omega}\phi\, d\mu_t^{\varepsilon}\leq \Cr{c-1} \|\phi\|_{C^2} + \int_{\partial \Omega}\nabla\phi\cdot\nu \Big(\frac{\varepsilon}{2}
|\nabla u^{\varepsilon}|^2+\frac{W(u^{\varepsilon})}{\varepsilon}\Big)\, d{\mathcal H}^{n-1}.
\end{equation*}
By integrating over $[t,t+1]$ and again using \eqref{engu2}, we obtain the desired estimate. 
\end{proof}
\section{Estimate on $\xi_t^{\e}$ from above}
\label{sec:4}
In this section we prove that $\xi_{t}^{\e}$ may be estimated from above
by the sup norm for any positive time.  One can prove the desired estimate
by modifying the similar estimate in \cite{Hutchinson,arXiv:1307.6629}
combined with the boundary behavior of $|\nabla u^{\e}|^2$ when $\Omega$
is strictly convex. It is here that the assumption of strict convexity is essential. 
%
\begin{proposition}[Negativity of the discrepancy] 
 \label{prop:4.1}
 For any $0<T<\infty$,  $0<\e<1$, there exists $\Cl[c]{c-6}$ depending only on 
 $T$ such that
 \begin{equation}
  \label{eq:4.2}
  \sup_{x\in \Omega,\, t\in [T,\infty)} \xi_t^{\e}   \leq \Cr{c-6}.
 \end{equation}
\end{proposition}

To show Proposition \ref{prop:4.1}, we use the following identities 
which gives a relationship between
the normal derivative of $|\nabla u|^2$ and the second fundamental form of the boundary.
Though it is a simple observation and has been used in a number of papers (see for
%
%
example \cite{MR0480282,MR0555661,MR1620498}), we include
the proof for the convenience of the reader. 
\begin{lemma}
 \label{lem:4.1}
 Let $B_x$ be the
 second fundamental form of $\partial\Omega$ at
 $x\in\partial\Omega$. Suppose that $u\in C^2(\overline{\Omega})$ satisfies
 $\nabla u\cdot\nu=0$ on $\partial\Omega$. Then at $x\in \partial\Omega$, we have
 \begin{equation}
  \frac{\partial}{\partial \nu}|\nabla u|^2
   =2B_x(\nabla u,\nabla u).
   \label{Bres}
 \end{equation}
\end{lemma}

\begin{remark}
 In particular, when $\Omega$ is convex, the right-hand side of \eqref{Bres} 
 is $\leq 0$. Furthermore, when $\Omega$ is strictly convex, \eqref{Bres} is $=0$ if and 
 only if $\nabla u=0$ at $x$. 
 \label{rem:4.1}
 \end{remark}
\begin{proof}
 [Proof of Lemma \ref{lem:4.1}]
 Without loss of generality by translation and rotation we may assume
 that $\partial\Omega$ is a graph near $x=0\in\partial\Omega$, namely
 there exists a function $f=f(x_1,\dots,x_{n-1})$ such that
 $\partial\Omega\cap B_r$ is included in the graph of $f$ for some $r>0$ and
 \begin{equation*}
  0=f(0,\cdots,0),\quad \nabla_{\R^{n-1}}f(0,\cdots,0)=0,\quad
   \frac{\partial^2f}{\partial x_i\partial x_j}(0,\cdots,0)=\kappa_j\delta_{ij},
 \end{equation*}
 where $\kappa_1,\dots,\kappa_{n-1}$ are the principal curvatures at
 $x=0$. We remark that
 \begin{equation*}
  B_0(X,Y)=\sum_{i,j=1}^{n-1}\frac{\partial^2f}{\partial x_i\partial x_j}(0)
 X_iY_j=\kappa_jX_jY_j
 \end{equation*}
for $X=(X_i)$, $Y=(Y_i)\in T_0\partial\Omega$. The outer unit normal
 vector is given by
 \begin{equation*}
    \nu=\frac{1}{\sqrt{1+|\nabla_{\R^{n-1}}f}|^2}(-\nabla_{\R^{n-1}}f,1).
 \end{equation*}
 By the boundary condition of $u$ we have
 \begin{equation*}
  0=\frac{\partial u}{\partial\nu}
   =\frac{1}{\sqrt{1+|\nabla_{\R^{n-1}}f}|^2}
   \left(
    -\sum_{i=1}^{n-1}
    \frac{\partial f}{\partial x_i}\frac{\partial u}{\partial x_i}
    +\frac{\partial u}{\partial x_n}
   \right).
 \end{equation*}
 Differentiating with respect to $x_j$ again and plugging in $x=0$, we have
 \begin{equation*}
  \frac{\partial^2 u}{\partial x_n\partial x_j}
   =\kappa_j\frac{\partial u}{\partial x_j}
 \end{equation*}
 for $j=1,\dots,n-1$.
By the boundary condition again, we may compute
 \begin{equation*}
  \frac{\partial}{\partial\nu}|\nabla u|^2
   =2\sum_{j=1}^{n-1}\frac{\partial u}{\partial x_j}
   \frac{\partial^2 u}{\partial x_n\partial x_j}
   =2\sum_{j=1}^{n-1}\kappa_j\frac{\partial u}{\partial
   x_j}\frac{\partial u}{\partial x_j}
   =2B_0(\nabla u,\nabla u).
 \end{equation*}
\end{proof}
In the proof of Lemma \ref{lem:4.1}, we also need the following.
\begin{lemma}
 \label{lem:6.2}
 There exists a constant $\Cl[c]{c-7}>0$ depending only on $\Omega$ such that
 \begin{equation*}
  \sup_{\Omega\times[\varepsilon^2,\infty)}
   \varepsilon|\nabla u^\varepsilon|\leq \Cr{c-7}
 \end{equation*}
 for all $1>\varepsilon>0$.
\end{lemma}
\begin{proof}
 [Proof of Lemma \ref{lem:6.2}]
 After the parabolic re-scaling, 
 the interior estimates for $\nabla u^{\e}$ can be obtained by
 the standard argument (see
 Lady{\v{z}}enskaja-Solonnikov-Ural'ceva~\cite{MR0241822}). 
 To show
 the boundary estimates for $\nabla u^{\e}$, we use the reflection
 argument on the tubular neighborhood of the boundary. A reflection
 of $u^{\e}$ satisfies a parabolic
 equation on the tubular neighborhood hence we may apply the interior
 estimates.
 \end{proof}
\begin{proof}
 [Proof of Proposition \ref{prop:4.1}]
  Under the parabolic change of variables
$x\mapsto \frac{x}{\e}$ and $t\mapsto\frac{t}{\e^2}$,
 we continue to use the same notation $u^{\e}$ which we denote by $u$ in the following.
  For $G$ to be chosen, define
  \begin{equation}
  \xi:=\frac{|\nabla u|^2}{2}-W(u)-G(u).
  \label{eq:4.0}
  \end{equation}
  We compute $\partial_t \xi-\Delta\xi$ and obtain
  \begin{equation}
  \begin{split}
  \partial_t\xi-\Delta\xi=&\nabla u\cdot\nabla\partial_t u-(W'+G')\partial_t u-|\nabla^2 u|^2
  -\nabla u\cdot \nabla\Delta u \\ & +(W'+G')\Delta u+(W''+G'')|\nabla u|^2.
  \end{split}
  \label{eq:4.01}
  \end{equation}
  Differentiate the equation \eqref{eq:1.1} after the change of variables with respect to $x_j$,
  multiply $\partial_{x_j}u$ and sum over $j$ to obtain
  \begin{equation}
  \nabla u\cdot\nabla\partial_t u = \nabla u\cdot \nabla\Delta u-W''|\nabla u|^2.
  \label{eq:4.02}
  \end{equation}
  By \eqref{eq:1.1}, \eqref{eq:4.01} and \eqref{eq:4.02}, we obtain
  \begin{equation}
  \partial_t\xi-\Delta\xi=W'(W'+G')-|\nabla^2 u|^2+G''|\nabla u|^2.
  \label{eq:4.03}
  \end{equation}
  Differentiating \eqref{eq:4.0} with respect to $x_j$ and by using the Cauchy-Schwarz
  inequality we have
  \begin{equation}
  \begin{split}
  \sum_{j=1}^n&\big(\sum_{i=1}^n \partial_{x_i} u\,\partial_{x_i x_j} u\big)^2
  =\sum_{j=1}^n (\partial_{x_j}\xi+(W'+G')\partial_{x_j} u)^2 \\
  &=|\nabla\xi|^2+2(W'+G')\nabla\xi\cdot\nabla u+(W'+G')^2 |\nabla u|^2
  \leq |\nabla u|^2|\nabla^2 u|^2.
  \end{split}
  \label{eq:4.04}
  \end{equation}
  On $\{|\nabla u|\neq 0\}$, divide \eqref{eq:4.04} by $|\nabla u|^2$ and substitute into
  \eqref{eq:4.03} to obtain
  \begin{equation}
  \partial_t \xi-\Delta\xi
  \leq -(G')^2-W'G'-\frac{2(W'+G')}{|\nabla u|^2}\nabla\xi\cdot\nabla u
  +G''|\nabla u|^2.
  \label{eq:4.05}
  \end{equation}
  Given $T>0$, by Lemma \ref{lem:6.2}, we have a uniform 
  estimate on $M:=\sup_{t\geq \e^{-2}T/2,\, x\in \e^{-1}\Omega}
   \frac{|\nabla u|^2}{2}$ depending only on $T$ but 
   independent of $0<\e<1$.  
 Let $\phi$ be a smooth function of $t$ such that $\phi(t)= M$ 
 for $t\leq \e^{-2}T/2$, $\phi(t)=0$ for $t\geq \e^{-2}T$, $0\leq
 \phi(t)\leq M$ for all $t$ and $|\phi'|\leq 4 T^{-1}\e^2 M$. 
 Let
 \begin{equation} 
 \tilde\xi:=\xi-\phi\ \ \mbox{and} \ \ G(u):=\e \Big(1-\frac{1}{8}(u-\gamma)^2
 \Big),
\label{eq:4.06}
\end{equation}
where $\gamma$ is as in \eqref{wcond1}. Due to the choice of $G$, we have
\begin{equation}
0<G<\e, \ \ G'W' \geq 0, \ \ G''=-\frac{\e}{4}
\label{eq:4.07}
\end{equation}
for $|u|\leq 1$. Now consider the maximum point of $\tilde\xi$ on $\e^{-1}\Omega
\times[\e^{-2}T/2, \tilde T]$ for any large $\tilde T$. Due to the choice of $M$ 
and $\phi$, $\tilde\xi\leq 0$ for $t=\e^{-2}T/2$. Suppose for a contradiction that
\begin{equation}
\max_{x\in \e^{-1}\overline\Omega,\, t\in [\e^{-2}T,\tilde T]} \xi\geq C \e
\label{eq:4.08}
\end{equation}
for some $C$ to be chosen. Since $\phi=0$ for $t\geq \e^{-2} T$, \eqref{eq:4.08}
implies
\begin{equation}
\max_{x\in \e^{-1}\overline\Omega,\, t\in [\e^{-2}T/2,\tilde T]} \tilde\xi\geq C\e.
\label{eq:4.09}
\end{equation}
Consider a maximum point $(\hat x,\hat t)$ of $\tilde\xi$ of \eqref{eq:4.09}. Note that 
$\hat x\notin \partial\Omega$. Because, if $\hat x\in\partial \Omega$, $\frac{\partial \xi}{\partial
\nu}\geq 0$ while Lemma \ref{lem:4.1} and Remark \ref{rem:4.1} show $\nabla u=0$.
But then $\xi<0$ there and we have a contradiction. Thus $\hat x$ is an interior point.
Furthermore, $\hat t>\e^{-2}T/2$ and thus we have 
\begin{equation}
\partial_t \tilde\xi\geq 0, \ \
\nabla \tilde \xi=\nabla \xi=0, \ \ \mbox{ and} \ \ \Delta\tilde\xi=\Delta\xi\leq 0
\label{eq:4.10}
\end{equation}
at $(\hat x,\hat t)$.
By evaluating \eqref{eq:4.05} at this point, and using \eqref{eq:4.07} and
\eqref{eq:4.10}, we obtain
\begin{equation}
-4T^{-1}\e^2 M\leq G'' |\nabla u|^2<-\frac{\e}{4}2C\e
\label{eq:4.11}
\end{equation}
where the last inequality follows from $|\nabla u|^2\geq 2\tilde\xi$. 
Thus choosing $C$ sufficiently large depending only on $T$ and $M$ which
ultimately depends only on $T$, we obtain a contradiction. Thus we proved
that 
\begin{equation}
\max_{x\in \e^{-1}\overline\Omega,\, t\in [\e^{-2}T/2,\tilde T]} \tilde \xi
\leq C\e.
\label{eq:4.12}
\end{equation}
Note that $\phi=0$ for $t\geq \e^{-2}T$ and $\tilde T$ is arbitrary, and since $G\leq \e$, we obtained the 
desired inequality \eqref{eq:4.2} by choosing $\Cr{c-6}:=C+1$. 
\end{proof}

\begin{corollary}
There exists 
$0<D_0<\infty$ depending only on $\Cr{c-1}, \Cr{c-2}$ and $T$ such that
\begin{equation}
\mu_t^{\e}(B_r (y)\cap \Omega)\leq D_0 r^{n-1}
\label{eq:4.13}
\end{equation}
for all $y\in \overline\Omega$, $0<r\leq\Cr{c-4}/4$, $0<\e<1$ and $t\geq T$. 
\label{densityupper}
\end{corollary}
\begin{proof}
Let $\Cr{c-6}$ be the constant in Proposition \ref{prop:4.1} corresponding to $T/2$.
Suppose $y\in N_{\Cr{c-2}/2}$. For $\hat t\geq T$ and $0<r\leq \Cr{c-2}/4$, 
set $s:=\hat t+r^2$ in the formulas of $\rho_1$ and $\rho_2$. We then integrate \eqref{eq:3.7}
over $t\in [\hat t-\frac{T}{2}, \hat t]$ to obtain 
\begin{equation}
e^{\Cr{c-3}(s-t)^{\frac14}} \int_{\Omega}
(\rho_1+\rho_2)\, d\mu_t^{\e}(x)\Big|_{t=\hat t -\frac{T}{2}}^{\hat t}
\leq e^{\Cr{c-3}(r^2+T/2)^{\frac14}}\Big( \frac{\Cr{c-4}T}{2}
+\Cr{c-6} \sqrt{4\pi} \int_{\hat t -\frac{T}{2}}^{\hat t} \frac{dt}{\sqrt{
s -t}}\Big)
\label{eq:4.14}
\end{equation}
where we have used \eqref{eq:4.2} and $\int_{{\mathbb R}^n}
\frac{\rho_i}{\sqrt{4\pi(s-t)}}\,dx\leq 1$. The right-hand side of
\eqref{eq:4.14} may be estimated in terms of a constant depending
only on $\Cr{c-1},\Cr{c-2}$ and $T$. Using $\eta(x-y)=1$
for $x\in B_r(y)$, we have
\begin{equation}
\begin{split}
\frac{e^{-\frac14}}{(4\pi)^{\frac{n-1}{2}} r^{n-1}} \mu_{\hat t}^{\e}(B_r(y)\cap \Omega)
& \leq \int_{B_r(y)\cap \Omega} \frac{e^{-{\frac14}}}{(4\pi r^2)^{\frac{n-1}{2}}}
\, d\mu_{\hat t}^{\e}(x) \\ & \leq 
\int_{B_r(y)\cap \Omega} \frac{e^{-\frac{|x-y|^2}{4r^2}}}{(4\pi r^2)^{\frac{n-1}{2}}}
\, d\mu_{\hat t}^{\e}(x)  \leq
e^{\Cr{c-3}(s-\hat t)^{\frac14}} \int_{\Omega}\rho_1\, d\mu_{\hat t}^{\e}.
\end{split}
\label{eq:4.15}
\end{equation}
On the other hand,
\begin{equation}
\begin{split}
e^{\Cr{c-3}(s-\hat t +\frac{T}{2})^{\frac14}}\int_{\Omega} (\rho_1+\rho_2)\, 
d\mu_{\hat t-\frac{T}{2}}^{\e}(x) & \leq 
e^{\Cr{c-3}(r^2+\frac{T}{2})^{\frac14}} \int_{\Omega} \frac{2}{(4\pi (r^2+\frac{T}{2}))^{
\frac{n-1}{2}}}\, d\mu_{\hat t-\frac{T}{2}}^{\e}(x) \\
&\leq 2 e^{\Cr{c-3}(\Cr{c-4}^2+T)^{\frac14}} (2\pi T)^{\frac{1-n}{2}} \Cr{c-1}.
\end{split}
\label{eq:4.16}
\end{equation}
Combining \eqref{eq:4.14}-\eqref{eq:4.16}, we obtain \eqref{eq:4.13} with 
an appropriate constant $D_0$ depending only on $\Cr{c-1},\Cr{c-2}$ and $T$.
The case of $y\in \Omega\setminus N_{\Cr{c-2}/2}$ can be proved using \eqref{eq:3.75}.
\end{proof}
\section{Convergence of the energy measures}
\label{sec:5}
In this section we prove that there exists a family of Radon measures
$\{\mu_t\}_{t\geq0}$ such that after taking some subsequence,
$\mu_t^{\varepsilon_i}\rightharpoonup\mu_t$ as $i\rightarrow\infty$ for
all $t\geq0$ on $\overline\Omega$. Note that we want  consider up to the
boundary convergence of $\mu_t^{\varepsilon}$, so we take a test function
which does not vanish near $\partial \Omega$ in general. 

\begin{lemma}[Semidecreasing properties]
 \label{lem:5.1}
 For all $\phi\in C^2(\overline\Omega)$ with $\phi\geq0$ on $\overline\Omega$, 
 we have
  \begin{equation*}
  \int_\Omega\phi\,d\mu_t^\varepsilon-\Cr{c-1}\|\phi\|_{C^2(\overline\Omega)} t
 \end{equation*}
 is monotone decreasing with respect to $t\geq0$ for all $0<\varepsilon<1$.
\end{lemma}

\begin{proof}
 [Proof of Lemma \ref{lem:5.1}]
 For $\phi$ with the given assumptions, using the Neumann condition of $u^{\varepsilon}$, we have
 \begin{equation*}
  \begin{split}
   \frac{d}{dt}\int_\Omega\phi\,d\mu_t^\varepsilon
   &=-\int_\Omega\varepsilon
   \nabla\phi\cdot\nabla u^\varepsilon
   \partial_tu^\varepsilon\,dx 
   -\int_\Omega
   \varepsilon\phi(\partial_tu^\varepsilon)^2\,dx \\
   &=\int_\Omega
   \frac{\varepsilon(\nabla\phi\cdot\nabla u^\varepsilon)^2}{4\phi}\,dx
   -\int_\Omega\varepsilon\phi \left(\partial_tu^\varepsilon+\frac{\nabla\phi\cdot\nabla
   u^\varepsilon}{2\phi}\right)^2\,dx \\
   &\leq\int_\Omega\frac{|\nabla\phi|^2}{2\phi}\,d\mu_t^\varepsilon 
   \leq\|\phi\|_{C^2(\overline\Omega)}\Cr{c-1}
  \end{split}
 \end{equation*}
 by \eqref{engu2}.
\end{proof}

 \begin{proposition}
  \label{prop:5.1}
  There exist a family of Radon measures $\{\mu_t\}_{t\geq0}$ and a
  subsequence such that
  $\mu_t^{\varepsilon_i}\rightharpoonup\mu_t$ as $i\rightarrow\infty$
  for all $t\geq0$ on $\overline\Omega$.
 \end{proposition}

\begin{proof}
 [Proof of Proposition \ref{prop:5.1}]
 Since we aim to obtain convergence of measures on $\overline\Omega$, we may
 define $\mu_t^{\varepsilon}$ to be zero measure on ${\mathbb R}^n\setminus
 \overline\Omega$ and we may regard $\mu_t^{\varepsilon}$ to be a measure on ${\mathbb R}^n$.
 Let $B_0\subset[0,\infty)$ be a countable, dense subset. Then by the
 compactness of Radon measures and the diagonal argument, there exist a
 family of Radon measures $\{\mu_t\}_{t\in B_0}$ and a subsequence such that
 $\mu_t^{\varepsilon_i}\rightharpoonup \mu_t$  as $i\rightarrow\infty$
 for $t\in B_0$ on ${\mathbb R}^n$. Obviously, $\mu_t$ has a support in $\overline\Omega$
 and note that it may be possible that $\mu_t(\partial\Omega)>0$ in general.

 Let $\{\phi_k\}_{k=1}^\infty\subset C^2(\overline\Omega)$ be a dense subset in
 $C(\overline\Omega)$. Then for each $k\in\N$, there is a countable set
 $B_k\subset[0,\infty)$ such that $\mu_t(\phi_k)$ has continuous
 extension with respect to $t\in[0,\infty)\setminus B_k$ by the
 semidecreasing property of $\mu_t(\phi_k)$. Therefore letting $B=\cup_{k=1}^\infty B_k$,
 which is countable, $\mu_t(\phi_k)$ is continuous extension with
 respect to $t\in[0,\infty)\setminus B$, namely for
 $s\in[0,\infty)\setminus B$, we may define
 \begin{equation}
  \label{eq:5.1}
   \lim_{\substack{t\uparrow s \\ t\in B_0}}\mu_t(\phi_k)
   =\lim_{\substack{t\downarrow s \\ t\in B_0}}\mu_t(\phi_k)
   =:\mu_s(\phi_k).
 \end{equation}

 Let $s\in[0,\infty)\setminus B$
 and let $\{\varepsilon_{i_j}\}_{j=1}^\infty$ be any subsequence
 satisfying 
 \begin{equation}
  \label{eq:5.2}
  \mu_s^{\varepsilon_{i_j}}\rightharpoonup
  \tilde{\mu}_s\quad\text{as}\quad
  j\rightarrow\infty
 \end{equation}
 for some Radon measure $\tilde{\mu}_s$. Then for any $t,t'\in B_0$ with
 $t<s<t'$ and for any $k\in\N$, we have
 \begin{equation*}
  \mu_t^{\varepsilon_{i_j}}(\phi_k)-\Cr{c-1}\|\phi_k\|_{C^2}(t-s)
   \geq\mu_s^{\varepsilon_{i_j}}(\phi_k)
   \geq\mu_{t'}^{\varepsilon_{i_j}}(\phi_k)-\Cr{c-1}\|\phi_k\|_{C^2}(t'-s).
 \end{equation*}
 From \eqref{eq:5.1} and \eqref{eq:5.2}, we have
 \begin{equation*}
  \mu_t(\phi_k)-\Cr{c-1}\|\phi_k\|_{C^2}(t-s)
   \geq\tilde\mu_s(\phi_k)
   \geq\mu_{t'}(\phi_k)-\Cr{c-1}\|\phi_k\|_{C^2}(t'-s)
 \end{equation*}
 hence taking $t\uparrow s$ and $t'\downarrow s$, we find
 $\tilde\mu_s(\phi_k)=\mu_s(\phi_k)$. Therefore
 $\mu_s^{\varepsilon_i}(\phi_k)$ converges to $\mu_s(\phi_k)$ as
 $i\rightarrow\infty$ for all $s\in[0,\infty)\setminus B$. Since
 $\{\phi_k\}_{k=1}^\infty$ is a dense subset in $C(\overline\Omega)$,
 $\mu_s^{\varepsilon_i}\rightharpoonup\mu_s$ as $i\rightarrow\infty$ for
 all $s\in [0,\infty)\setminus B$.

 Finally since $B$ is countable, we may choose a further subsequence (denoted by
 same index) such that $\mu_t^{\varepsilon_i}$ converges to some Radon
 measure $\mu_t$ for all $t\geq0$ by the diagonal argument.
\end{proof}

\section{Vanishing of the discrepancy}
\label{sec:6}
In this section, we prove the vanishing of $L^1$ limit of $|\xi^{\e_i}_t|$ as a sequence of functions 
on $\overline\Omega\times(0,\infty)$. Note that, due to \eqref{engu2} and the 
weak compactness theorem of Radon measures, we may 
choose a subsequence (denoted by the same index) 
such that $|\xi^{\e_i}_t|\, dxdt$ converges to a Radon measure on $\overline\Omega\times[0,\infty)$
locally in time. We show that the limit measure denoted by $|\xi|$ is identically 0, which will
prove the $L^1$ vanishing. 
We also define $d\mu^\varepsilon:=d\mu_t^\varepsilon dt$ and 
the subsequence limit $\mu$ on $\overline\Omega\times[0,\infty)$.
\begin{lemma}
 \label{lem:6.1}
 For any $(x',t')\in\spt\mu$ with $t'>0$ and $x'\in \overline\Omega$,
 there exist a sequence $\{(x_i,t_i)\}_{i=1}^\infty$ and a subsequence
 $\varepsilon_i$ (denoted by same index) such that $t_i>0$, $x_i\in
 \Omega$, $(x_i,t_i)\rightarrow(x',t')$ as $i\rightarrow\infty$ and
 $|u^{\varepsilon_i}(x_i,t_i)|<\alpha$ for all $i\in\N$.
\end{lemma}

\begin{proof}
 [Proof of Lemma \ref{lem:6.1}]
 For simplicity we omit the subscript $i$. For a contradiction, assume that there exists
 $0<r_0<\sqrt{t'}$ such that 
 \begin{equation}
  \label{eq:6.4}
  \inf_{(B_{r_0}(x')\cap \Omega)\times(t'-r_0^2,t'+r_0^2)}|u^{\varepsilon}|>\alpha
 \end{equation}
 for all sufficiently small $\varepsilon>0$.
 Differentiating \eqref{eq:1.1} with respect to $x_j$, we have
 \begin{equation}
  \label{eq:6.1}
  \varepsilon\partial_t(\partial_{x_j}u^\varepsilon)
   =\varepsilon\Delta(\partial_{x_j}u^\varepsilon)
   -\frac{W''(u^\varepsilon)}{\varepsilon}\partial_{x_j}u^\varepsilon.
 \end{equation}
 Fix $\phi\in C_c^2(B_{r_0}(x'))$ such that 
 \begin{equation*}
  |\nabla\phi|\leq\frac{3}{r_0},\quad \phi\big|_{B_{\frac{r_0}{2}}(x')}\equiv 1.
 \end{equation*}
 Then testing $\partial_{x_j}u^\varepsilon\phi^2$ to \eqref{eq:6.1}, we
 have
 \begin{equation*}
   \begin{split}
    \frac{d}{dt}\int_\Omega
    \frac{\varepsilon}{2}|\nabla u^\varepsilon|^2\phi^2\,dx
    &=-\varepsilon\int_\Omega|\nabla^2u^\varepsilon|^2\phi^2\,dx 
    -2\varepsilon\int_\Omega\partial_{x_j}u^\varepsilon\phi
    (\nabla\partial_{x_j}u^\varepsilon\cdot\nabla\phi)\,dx \\
    &\qquad+\varepsilon\int_{\partial\Omega}
    \partial_{x_j}u^\varepsilon\phi^2
    (\nabla\partial_{x_j}u^\varepsilon\cdot\nu)\,d\sigma 
    -\frac1\varepsilon\int_\Omega
    W''(u^\varepsilon)|\nabla u^\varepsilon|^2\phi^2\,dx.
   \end{split} 
\end{equation*}
 By the H\"older and Young inequalities, Lemma \ref{lem:4.1} and the
convexity of $\Omega$, we have
 \begin{equation*}
   \frac{d}{dt}\int_\Omega
   \frac{\varepsilon}{2}|\nabla u^\varepsilon|^2\phi^2\,dx
   \leq
   -\frac1\varepsilon\int_\Omega
   W''(u^\varepsilon)|\nabla u^\varepsilon|^2\phi^2\,dx
   +\varepsilon\int_{\Omega}
   |\nabla u^{\varepsilon}|^2|\nabla\phi|^2\,dx.
 \end{equation*}
 Using \eqref{eq:6.4} and \eqref{wcond2}, we have
 \begin{equation}
  \label{eq:6.3}
   \frac{d}{dt}\int_\Omega
  \frac{\varepsilon}{2}|\nabla u^\varepsilon|^2\phi^2\,dx
   \leq
   -\frac{2\kappa}{\varepsilon^2}\int_\Omega
   \frac\varepsilon2|\nabla u^\varepsilon|^2\phi^2\,dx
   +\frac{18}{r_0^2}\Cr{c-1}.
 \end{equation}
 Applying the Gronwall inequality to \eqref{eq:6.3}, we obtain
 \begin{equation*}
  \int_\Omega
   \frac{\varepsilon}{2}|\nabla u^\varepsilon|^2\phi^2\,dx
  \leq
  \left(
   \exp\left({\frac{2\kappa}{\varepsilon^2}(t'-r_0^2-t)}\right)
   +\frac{9\varepsilon^2}{r_0^2\kappa}\right)\Cr{c-1}
 \end{equation*}
 for $t'-r_0^2<t<t'+r_0^2$ hence
 \begin{equation}
  \label{eq:6.5}
  \int_{t'-r_0^2}^{t'+r_0^2}\,dt
   \int_\Omega
   \frac{\varepsilon}{2}|\nabla u^\varepsilon|^2\phi^2\,dx
   \rightarrow 0
   \quad \text{as}\ \varepsilon\downarrow0.
 \end{equation}

 By the continuity of $u^\varepsilon$ and \eqref{eq:6.4}, we may assume
 $\alpha\leq u^\varepsilon\leq1$ on
 $(B_{r_0}(x')\cap\Omega)\times(t'-r_0^2,t'+r_0^2)$ without loss of
 generality. Otherwise we have $-1\leq u^\varepsilon\leq-\alpha$ and we
 may argue similarly. Testing $(u^\varepsilon-1)\phi^2$ on
 $\Omega\times(t'-r_0^2,t'+r_0^2)$ to \eqref{eq:1.1} we have
 \begin{equation*}
\begin{split}
 &\quad\frac{\varepsilon}2\int_\Omega(u^\varepsilon-1)^2\phi^2\,dx\bigg|_{t=t'+r_0^2}
 -\frac{\varepsilon}2\int_\Omega(u^\varepsilon-1)^2\phi^2\,dx\bigg|_{t=t'-r_0^2} \\
 &\leq\varepsilon\int_{t'-r_0^2}^{t'+r_0^2}\,dt
 \int_\Omega(u^\varepsilon-1)^2|\nabla\phi|^2\,dx
 -\frac1\varepsilon\int_{t'-r_0^2}^{t'+r_0^2}\,dt
 \int_\Omega W'(u^\varepsilon)(u^\varepsilon-1)\phi^2\,dx
 \end{split} 
\end{equation*}
 hence
 \begin{equation*}
  \begin{split}
   \int_{t'-r_0^2}^{t'+r_0^2}\,dt
   \int_\Omega
   \frac{W'(u^\varepsilon)}{\varepsilon}(u^\varepsilon-1)\phi^2\,dx
   &\leq\varepsilon\int_{t'-r_0^2}^{t'+r_0^2}\,dt 
   \int_\Omega(u^\varepsilon-1)^2|\nabla\phi|^2\,dx \\
   &\quad+\frac{\varepsilon}2\int_\Omega(u^\varepsilon-1)^2\phi^2\,dx\bigg|_{t=t'-r_0^2}.
  \end{split}   
 \end{equation*}
 Using
 \begin{equation*}
  W'(s)(s-1)
   \geq\kappa(s-1)^2
   \geq cW(s)
 \end{equation*}
 for some constant $c>0$ if $\alpha\leq s\leq1$, we may obtain
  \begin{equation*}
  \begin{split}
   \int_{t'-r_0^2}^{t'+r_0^2}\,dt
   \int_\Omega
   \frac{W(u^\varepsilon)}{\varepsilon}\phi^2\,dx
   &\leq\frac\varepsilon{c}\int_{t'-r_0^2}^{t'+r_0^2}\,dt 
   \int_\Omega(u^\varepsilon-1)^2|\nabla\phi|^2\,dx \\
   &\quad+\frac{\varepsilon}{2c}\int_\Omega(u^\varepsilon-1)^2\phi^2\,dx\bigg|_{t=t'-r_0^2}
  \end{split}   
 \end{equation*}
 hence
 \begin{equation}
  \label{eq:6.6}
   \int_{t'-r_0^2}^{t'+r_0^2}\,dt
   \int_\Omega
   \frac{W(u^\varepsilon)}{\varepsilon}\phi^2\,dx
   \rightarrow 0\quad\text{as}\ \varepsilon\downarrow0.
 \end{equation}
 Thus we have by \eqref{eq:6.5} and \eqref{eq:6.6}
 \begin{equation*}
  \int_{t'-r_0^2}^{t'+r_0^2}
   \left(
    \int_\Omega\phi^2\,d\mu_t^\varepsilon
   \right)
   \,dt\rightarrow0 \quad\text{as}\ \varepsilon\downarrow0.
 \end{equation*}
This shows that $(x',t')\notin\spt\mu$, which is contradiction. 
\end{proof}

\begin{lemma}
 \label{lem:6.3}
 There exist $\delta_0$, $r_0$, $\gamma_0>0$ depending only on $\kappa$,
 $W$ and $T>0$ such that the following holds:
 If
 \begin{equation}
  \label{eq:6.9}
  \int_{\overline\Omega}\eta(x-y)\rho_{(y,s)}(x,t)\,d\mu_s(y)<\delta_0
 \end{equation}
 for some $T<t<s<t+\frac{r_0^2}{2}$ and $x\in \overline\Omega$, then
 $(x',t')\not\in\spt\mu$ for all $x'\in B_{\gamma_0r}(x)\cap
 \overline\Omega$, where $t'=2s-t$ and $r=\sqrt{2(s-t)}$.
\end{lemma}

\begin{proof}
 [Proof of Lemma \ref{lem:6.3}]
 In the following we assume $x'\in N_{\Cr{c-2}/2}$. The proof for the
 case $x'\in \Omega\setminus N_{\Cr{c-2}/2}$ may be carried out using
 $\eqref{eq:3.75}$ in place of $\eqref{eq:3.7}$.  Let us assume
 $(x',t')\in\spt\mu$ for a contradiction. Then by Lemma \ref{lem:6.1}
 there exists a sequence $\{(x_i,t_i)\}_{i=1}^\infty$ such that
 $(x_i,t_i)\rightarrow (x',t')$ as $i\rightarrow\infty$ and
 $|u^{\varepsilon_i}(x_i,t_i)|<\alpha$ for all $i\in\N$. Put
 $r_i:=\gamma_0\varepsilon_i$, where $\gamma_0>0$ will be chosen later,
 and $T_i:=t_i+r_i^2$. Then
 \begin{equation}
  \label{eq:6.6s}
   \begin{split}
    &\quad\int_{B_{r_i}(x_i)}
    \eta(y-x_i)\rho_{(x_i,T_i)}(y,t_i)\,d\mu_{t_i}^{\varepsilon_i}(y) \\
    &\geq
    \frac{1}{(4\pi r_i^2)^\frac{n-1}{2}}    
    \int_{B_{r_i}(x_i)}
    \eta(y-x_i)
    \exp\left(-\frac{|y-x_i|^2}{4r_i^2}\right)
    \frac{W(u^{\varepsilon_i}(y,t_i))}{\varepsilon_i}
    \,dy.
   \end{split} 
 \end{equation}
 For $y\in B_{r_i}(x_i)$,
 \begin{equation*}
  |u^{\varepsilon_i}(y,t_i)|
   \leq \gamma_0\sup_{x\in\Omega}
   \varepsilon_i|\nabla u^{\varepsilon_i}(x,t_i)|
   +|u^{\varepsilon_i}(x_i,t_i)|
   \leq \Cr{c-7}\gamma_0+\alpha,
 \end{equation*}
 where $\Cr{c-7}$ is a constant given by Lemma \ref{lem:6.2}. Thus for
 sufficiently small $\gamma_0>0$ and $y\in B_{r_i}(x_i)$, we have
 $W(u^{\varepsilon_i}(y,t_i))\geq c$ for some $c>0$. Thus for all
 sufficiently large $i$, we may obtain from \eqref{eq:6.6s}
 \begin{equation*}
  \int_{B_{r_i}(x_i)}
  \eta(y-x_i)\rho_{(x_i,T_i)}(y,t_i)\,d\mu_{t_i}^{\varepsilon_i}(y) \\
  \geq
   \frac{c}{(4\pi \gamma_0^2)^\frac{n-1}{2}\varepsilon_i^n}    
   \int_{B_{r_i}(x_i)}
   \exp\left(-\frac{|y-x_i|^2}{4r_i^2}\right)\,dy
   \geq \Cl[c]{c-8}
 \end{equation*}
 for some constant $\Cr{c-8}>0$.
 By \eqref{eq:3.7} and \eqref{eq:4.2} we have
 \begin{equation*}
  \begin{split}
   \Cr{c-8}
   &\leq
   \int_{\Omega}
   (\eta(y-x_i)\rho_{(x_i,T_i)}(y,t_i)
   +\eta(\tilde{y}-x_i)\tilde\rho_{(x_i,T_i)}(y,t_i))
   \,d\mu_{t_i}^{\varepsilon_i}(y) \\
   &\leq
   e^{
   \Cr{c-3}(T_i-s)^\frac14} 
   \int_{\Omega}
   (\eta(y-x_i)\rho_{(x_i,T_i)}(y,s)
   +\eta(\tilde{y}-x_i)\tilde\rho_{(x_i,T_i)}(y,s))   
   \,d\mu_{s}^{\varepsilon_i}(y) \\
   &\quad +\int_s^{t_i}
 e^{
   \Cr{c-3}(T_i-\tau)^\frac14}\Big(\Cr{c-4}+\frac{\sqrt{4\pi}\Cr{c-6}}{\sqrt{T_i-\tau}}\Big)
   \,d\tau.
  \end{split} 
 \end{equation*}
 Letting $i\rightarrow\infty$, we have
 \begin{equation}
  \begin{split}
   \Cr{c-8}
   &\leq
   e^{
   \Cr{c-3}(t'-s)^\frac14}
   \int_{\overline\Omega}
   (\eta (y-x')\rho_{(x',t')}(y,s)
   +\eta(\tilde{y}-x')\tilde\rho_{(x',t')}(y,s))   
   \,d\mu_{s}(y) \\
   &\quad +\int_s^{t'}
   e^{
   \Cr{c-3}(t'-\tau)^\frac14}\Big(\Cr{c-4}+\frac{\sqrt{4\pi}\Cr{c-6}}{\sqrt{t'-\tau}}\Big)\,d\tau.
  \end{split} 
  \label{eq:6s1}
 \end{equation}
 Since $t'-s=s-t=\frac{r^2}{2}$, we may choose sufficiently
 small $r_0$ such that $s-t<r_0^2/2$ implies
 \begin{equation}
  \int_s^{t'}
   e^{
   \Cr{c-3}(t'-\tau)^\frac14}\Big(\Cr{c-4}+\frac{\sqrt{4\pi}\Cr{c-6}}{\sqrt{t'-\tau}}\Big)\,d\tau
   \leq\frac{\Cr{c-8}}2, \ \ 
   e^{\Cr{c-3}(t'-s)^\frac14}
   \leq 2.
   \label{eq:6s2}
 \end{equation}
 By the convexity of $\Omega$, we have $|y-x'|\leq |{\tilde y}-x'|$ for $y,{\tilde y}\in B_{\Cr{c-2}/2}
(x')\subset N_{\Cr{c-2}}$, thus
 considering \eqref{etadef} as well, we have 
 \begin{equation}
   \eta(\tilde{y}-x')\tilde\rho_{(x',t')}(y,s)   
   \leq
   \eta(y-x')\rho_{(x',t')}(y,s).
   \label{eq:6s3}
  \end{equation}
Combining \eqref{eq:6s1}-\eqref{eq:6s3} and putting $\delta_0:=\frac{\Cr{c-8}}{32}$, we have
 \begin{equation}
  \label{eq:6.8}
   4\delta_0
   \leq\int_{\overline\Omega}
   \eta(y-x')\rho_{(x',t')}(y,s)
  \,d\mu_{s}(y).
 \end{equation}

 Now we assume \eqref{eq:6.9}. Then for any $\delta>0$ we may take
 $\gamma_1>0$ as in Lemma \ref{lem:A.1} (note also Corollary \ref{densityupper}) such that
 \begin{equation*}
  \begin{split}
   \int_{\overline\Omega}
   \eta(y-x')\rho_{(x',t')}(y,s)
   \,d\mu_{s}(y)
   &=\int_{\overline\Omega}
   \eta(y-x')\rho_{x'}^{r}(y)
   \,d\mu_{s}(y) \\
   &\leq(1+\delta)
   \int_{\overline\Omega}
   \eta(y-x)\rho_{x}^{r}(y)
   \,d\mu_{s}(y)
   +\delta D_0 \\
   &=(1+\delta)
   \int_{\overline\Omega}
   \eta(y-x)\rho_{(y,s)}(x,t)
   \,d\mu_{s}(y)
   +\delta D_0 \\
   &\leq
   \delta_0(1+\delta)+\delta D_0.
  \end{split} 
\end{equation*}
 Choose $\delta>0$ such that
 $\delta_0(1+\delta)+\delta D_0\leq 2\delta_0$. Then we have from
 \eqref{eq:6.8}
 \begin{equation*}
  4\delta_0
   \leq\int_{\overline\Omega}
   \eta(y-x')\rho_{(x',t')}(y,s)
   \,d\mu_{s}(y)
   \leq2\delta_0,
 \end{equation*}
 which is contradiction. Hence we have $(x',t')\not\in\spt\mu$.
\end{proof}

\begin{lemma}[Forward density lower bounds]
 \label{lem:6.4}
 For $T>0$, let $\delta_0(T)>0$ be a constant given in Lemma \ref{lem:6.3}. Then we have
 $\mu (Z^-(T))=0$, where
 \begin{equation*}
   Z^-(T):=\left\{(x,t)\in \spt\mu:
   \limsup_{s\downarrow t}\int_{\overline\Omega}
   \eta(y-x)\rho_{(y,s)}(x,t)\,d\mu_s(y)<\delta_0(T),\ \ t>T\right\}.
 \end{equation*}
\end{lemma}

\begin{proof}
 [Proof of Lemma \ref{lem:6.4}]
 We do not write out the dependence on $T$ in the following for simplicity, where
 we assume $t>T$.
 Corresponding to $T$, let $\delta_0$, $\gamma_0$ and $r_0$ be constants given by 
 Lemma \ref{lem:6.3}. For $0<\tau<\frac{r_0^2}{2}$ define
 \begin{equation*}
  Z^{\tau}
   :=
   \left\{
    (x,t)\in \spt\mu:
    \int_{\overline\Omega}\eta(y-x)\rho_{(y,s)}(x,t)\,d\mu_s(y)<\delta_0
    \quad\text{for}\ t<s<t+\tau
   \right\}.
 \end{equation*}
 If we take a sequence $\tau_m>0$ with $\lim_{m\rightarrow\infty}\tau_m= 0$, 
 then $Z^-\subset \cup_{m=1}^{\infty}Z^{\tau_m}$. Hence we only need to
 show $\mu(Z^{\tau})=0$. 

 Let $(x,t)\in Z^{\tau}$ be fixed and we define
 \begin{equation*}
  P(x,t)
   :=\left\{(x',t')\in\overline\Omega\times(0,\infty):
      \gamma_0^{-2}|x'-x|^2<|t'-t|<\tau\right\}. 
 \end{equation*}
 We claim that $P(x,t)\cap Z^{\tau}=\emptyset$. Indeed,
 suppose for a contradiction that $(x',t')\in P(x,t)\cap Z^{\tau}$. 
 Assume $t'>t$ and put $s=\frac12(t+t')$. Then
 $t<s\leq t+\tau$, $|x-x'|<\gamma_0\sqrt{|t'-t|}=\gamma_0\sqrt{2(s-t)}$
 and
 \begin{equation*}
  \int_{\overline\Omega}\eta(y-x)\rho_{(y,s)}(x,t)\,d\mu_s(y)<\delta_0.
 \end{equation*}
 Hence by Lemma \ref{lem:6.3},
 $(x',t')\not\in\spt\mu$, which contradicts $(x',t')\in
 Z^{\tau}$. If $t'< t$, by the similar
 argument, we obtain $(x,t)\notin
 \spt\mu$ which is a contradiction. This proves $P(x,t)\cap Z^{\tau}=\emptyset$.

 For a fixed $(x_0,t_0)\in\overline\Omega\times(T,\infty)$, define
 \begin{equation*}
 Z^{\tau,x_0,t_0}
 :=Z^{\tau}\cap\left(B_{\frac{\gamma_0}2\sqrt{\tau}}(x_0)\times(t_0-\frac{\tau}{2},t_0+\frac{\tau}{2})\right).
 \end{equation*}
 Then $Z^{\tau}$ is a countable union of $Z^{\tau,x_m,t_m}$ with
 $(x_m,t_m)$ spaced appropriately. Hence we only need to show that
 $\mu(Z^{\tau,x_0,t_0})=0$. Denote $Z^{\tau,x_0,t_0}$ by $Z'$. For
 $0<\rho\leq 1$, we may find a covering of $\pi_\Omega(Z'):=\{x\in
 \overline\Omega\,:\, (x,t)\in Z'\}$ by a collection of balls
 $\{B_{r_i}(x_i)\}_{i=1}^\infty$, where $(x_i,t_i)\in Z'$,
 $r_i\leq\rho$, so that
 \begin{equation}
  \sum_{i=1}^\infty\omega_nr_i^n
   \leq c(n)\mathscr{L}^n(B_{\frac{\gamma_0}{2}\sqrt{\tau}}(x_0))).
   \label{eq6.4sub}
 \end{equation}
 For for such covering, we find
 \begin{equation*}
  Z'\subset
   \bigcup_{i=1}^\infty
   B_{r_i}(x_i)\times\left(t_i-r_i^2\gamma_0^{-2},t_i+r_i^2\gamma_0^{-2}\right).
 \end{equation*}
 Indeed, if $(x,t)\in Z'$, then $x\in B_{r_i}(x_i)$ for some
 $i\in\N$. Since $P(x_i,t_i)\cap Z^{\tau}=\emptyset$, we
 have
 \begin{equation*}
  |t-t_i|\leq |x-x_i|^2\gamma_0^{-2}<r_i^2\gamma_0^{-2}.
 \end{equation*}

 Therefore we obtain by \eqref{eq6.4sub}
 \begin{equation*}
  \begin{split}
  \mu(Z')&\leq \sum_{i=1}^{\infty} \mu(B_{r_i}(x_i)\times(t_i-r_i^2 \gamma_0^{-2},t_i+r_i^2\gamma_0^{-2})) 
     \leq\sum _{i=1}^{\infty} 2r_i^2 \gamma_0^{-2} D_0 r_i^{n-1}  \\
     &\leq  2\rho \gamma_0^{-2}D_0 \omega_n^{-1}c(n)\mathscr{L}^n(B_{\frac{\gamma_0}{2}\sqrt{\tau}}(x_0)).
      \end{split}
 \end{equation*} 
 Since $\rho$ is arbitrary, we have $\mu(Z')=0$. This concludes the proof. 
\end{proof}

\begin{proposition}[Vanishing of discrepancy]
 \label{prop:6.1}
 We have $|\xi|=0$ on $\overline\Omega\times(0,\infty)$.
\end{proposition}

\begin{proof}
 [Proof of Proposition \ref{prop:6.1}]
 Due to \eqref{engu}, it is enough to prove $|\xi|=0$ on $\overline\Omega\times(T_1,T_2)$
 for all $0<T_1<T_2<\infty$. In the following we fix $T_1$ and $T_2$. 
 For $y\in N_{\Cr{c-2}/2}$ and $T_2>s>t>T_1$, by \eqref{eq:3.7} and \eqref{eq:4.2} we obtain
 ($\Cr{c-6}$ corresponding to $T_1$)
 \begin{equation*}
\frac{d}{dt}
   \Big(e^{\Cr{c-3}(s-t)^\frac14}
   \int_{\Omega}(\rho_1+\rho_2)\,d\mu^{\varepsilon_i}_t\Big)
   +e^{\Cr{c-3}(s-t)^{\frac14}} \int_{\Omega}\frac{\rho_1+\rho_2}{2(s-t)}\,
   d|\xi^{\e}_t|
    \leq 
   e^{\Cr{c-3}(s-t)^{\frac14}}\big(\Cr{c-4}+\frac{2\Cr{c-6}\sqrt{4\pi}}{(s-t)^{\frac12}}\big).
\end{equation*} 
Integrating over $t\in (T_1,s)$ and taking $i\rightarrow\infty$,
 we obtain
 \begin{equation}
  \iint_{\overline\Omega\times(T_1,s)}
  \frac{\rho_1+\rho_2}{2(s-t)}
   d|\xi|\leq e^{\Cr{c-3}s^{\frac14}}
   \int_{\overline\Omega}(\rho_1+\rho_2)\,d\mu_{T_1}
   +\int_{T_1}^s e^{\Cr{c-3}(s-t)^{\frac14}}\big(\Cr{c-4}+\frac{2\Cr{c-6}\sqrt{4\pi}}{(s-t)^{\frac12}}\big)\, dt.  
  \label{a1}
 \end{equation}
 Note that the right-hand side of \eqref{a1}
 is uniformly bounded for $(y,s)\in N_{\Cr{c-2}/2}\times(T_1,T_2)$ once 
 $T_1$ and $T_2$ are fixed. For $y\in \Omega\setminus N_{\Cr{c-2}/2}$, the 
 similar argument using \eqref{eq:3.75} in place of \eqref{eq:3.7} gives the 
 similar estimate (with $\rho_2=0$). Since the right-hand side of \eqref{a1}
 is bounded uniformly on $\overline\Omega\times(T_1,T_2)$, integration of \eqref{a1} over
 $(y,s)\in\overline\Omega\times(T_1,T_2)$ with respect to $d\mu_s ds$ shows
 that 
 \begin{equation}
  \label{eq:6.11}
  \int_{T_1}^{T_2}ds\int_{\overline\Omega}d\mu_s(y)
  \iint_{\overline\Omega\times(T_1,s)}
 \frac{\rho_1+\rho_2}{2(s-t)}
  d|\xi|(x,t) 
  \end{equation}
  is finite. 
 By the Fubini theorem, \eqref{eq:6.11} is turned
 into
 \begin{equation*}
   \iint_{\overline\Omega\times(T_1,T_2)}d|\xi|(x,t)
   \int_{t}^{T_2}ds\int_{\overline\Omega}
   \frac{\rho_1+\rho_2}{2(s-t)}
   \,d\mu_s(y).
\end{equation*}
Thus we have
 \begin{equation}
  \label{eq:6.12}
   \int_{t}^{T_2}\frac{1}{2(s-t)}\,ds
  \int_{\overline\Omega}
  \rho_1+\rho_2 \,d\mu_s(y)
  <\infty
 \end{equation}
 for $|\xi|$-almost all $(x,t)\in\overline\Omega\times(T_1,T_2)$. We next prove that
 for $|\xi|$-almost all $(x,t)$, 
 \begin{equation}
  \lim_{s\downarrow t}\int_{\overline\Omega}\rho_1\,d\mu_s(y)=0. 
  \label{eq:6.12s}
 \end{equation}
 For $t<s$, we define $\beta:=\log(s-t)$ and
 \begin{equation*}
  h(s):=\int_{\overline\Omega}\rho_1\,d\mu_s(y).
 \end{equation*}
 Then \eqref{eq:6.12} is translated into
 \begin{equation}
  \int_{-\infty}^{\log(T_2-s)}
  h(t+e^\beta)\,d\beta<\infty.
  \label{a2}
 \end{equation}

 Let $0<\theta<1$ be arbitrary for the moment. Due to \eqref{a2}, we may choose a decreasing
 sequence $\{\beta_i\}_{i=1}^\infty$ such that
 $\beta_i\rightarrow-\infty$, $\beta_i-\beta_{i+1}<\theta$ and
 \begin{equation*}
h(t+e^{\beta_i})<\theta
 \end{equation*}
 for all $i$. 
 For any $-\infty<\beta< \beta_1$ fixed, we may choose $i\geq 2$ such that
 $\beta_i\leq\beta<\beta_{i-1}$. We use $\rho_{(y,t+\e^{\beta})}(x,t)=\rho_{(x,t+2
 \e^{\beta})}(y,t+\e^{\beta})$ and use \eqref{eq:3.7} and \eqref{eq:4.2} to obtain
 \begin{equation}
  \begin{split}
  & h(t+e^\beta) 
    =\int_{\overline\Omega} \eta(y-x)\rho_{(y,t+e^{\beta})} (x,t)\, d\mu_{t+e^{\beta}}(y) \\
   &\leq
   \int_{\overline\Omega}\eta(x-y)\rho_{(x,t+2e^{\beta})}(y,t+e^\beta)
   +\eta(x-\tilde{y})\tilde\rho_{(x,t+2e^{\beta})}(y,t+e^\beta)
   \,d\mu_{t+e^{\beta}}(y)\\
   &\leq e^{\Cr{c-3}(2e^\beta-e^{\beta_i})^\frac14}
   \int_{\overline\Omega}\eta(x-y)
   \rho_{(x,t+2e^{\beta})}(y,t+e^{\beta_i}) 
   +\eta(x-\tilde{y})
   \tilde\rho_{(x,t+2e^{\beta})}(y,t+e^{\beta_i})
   \,d\mu_{t+e^{\beta_i}}(y) \\
   &+ \int_{t+e^{\beta_i}}^{t+e^{\beta}}
   e^{\Cr{c-3}(t+2e^{\beta}-\tau)^{\frac14}}\big(\Cr{c-4}+\frac{\Cr{c-6}\sqrt{4\pi}}{
   (t+2e^{\beta}-\tau)^{\frac12}}
   \big)\, d\tau.
  \end{split} 
  \label{eq:6.13b}
\end{equation}
Let us denote the last integral of \eqref{eq:6.13b} as $c(i)$. Note that $c(i)$ can be
made uniformly small (with respect to $i$) if $\theta$ is chosen small. 
By the convexity of $\Omega$, we have $|x-\tilde y|\geq  |x-y|$ for $x\in \overline\Omega$
and $y\in N_{\Cr{c-2}/2}$, thus
 \begin{equation*}
   \eta(x-\tilde{y})
   \tilde\rho_{(x,t+2e^{\beta})}(y,t+e^{\beta_i})
   \leq \eta(x-y)
   \rho_{(x,t+2e^{\beta})}(y,t+e^{\beta_i}).
 \end{equation*}
Hence we obtain
\begin{equation}
  \label{eq:6.13}  
    h(t+e^\beta)  
    \leq2e^{\Cr{c-3}(2R_{i}^2)^\frac14}
    \int_{\overline\Omega}\eta(x-y)
    \rho^{R_i}_{x}(y)
    \,d\mu_{t+e^{\beta_i}}(y) 
   +c(i)
\end{equation}
where $2R_i^2=2e^\beta-e^{\beta_i}$.

 We next show the lower bound of
 $h(t+e^{\beta_i})$. By the assumption of $\beta_i$,
 we have
 \begin{equation}
  \label{eq:6.14}
  \begin{split}
   \theta\geq h(t+e^{\beta_i})
   &=
   \int_{\overline\Omega}\eta(x-y)\rho_{(y,t+e^{\beta_i})}(x,t)
   \,d\mu_{t+e^{\beta_i}}(y) \\
   &=
   \int_{\overline\Omega}\eta(x-y)\rho^{r_{i}}_x(y)\,d\mu_{t+e^{\beta_i}}(y),
  \end{split}
 \end{equation}
where $2r_{i}^2=e^{\beta_i}$. Since $\beta\geq \beta_i$, we have $R_{i}\geq r_{i}$. Also $\beta-\beta_i<\beta_{i-1}-\beta_i<\theta$
implies $R_{i}^2/r_{i}^2<2e^{\theta}-1$ which can be made arbitrarily close to 1 by restricting $\theta$ to be small.
For arbitrary $\delta>0$, we restrict $\theta$ to be sufficiently small using Lemma \ref{lem:A.1} so that
$\frac{R_i}{r_i}<1+\gamma_2$, where $\gamma_2>0$ is
given by Lemma \ref{lem:A.1} corresponding to $\delta>0$. Then we obtain
\begin{equation}
 \label{eq:6.15}
 \int_{\overline\Omega}\eta(x-y)\rho^{R_{i}}_x(y)\,d\mu_{t+e^{\beta_i}}(y)
 \leq (1+\delta)\int_{\overline\Omega}\eta(x-y)\rho^{r_{i}}_x(y)\,d\mu_{t+e^{\beta_i}}(y)
 +\delta D_0
\end{equation}
hence from \eqref{eq:6.13}, \eqref{eq:6.14} and \eqref{eq:6.15} we have
\begin{equation*}
  \begin{split}
    h(t+e^\beta) 
   &\leq2e^{\Cr{c-3}(2R_i^2)^{\frac14}}\Big(
   (1+\delta)\int_{\overline\Omega}\eta(x-y)
   \rho^{r_{i}}_{x}(y)
   \,d\mu_{t+e^{\beta_i}}(y)
   +\delta D_0\Big) +c(i) \\
   &\leq2e^{\Cr{c-3}(2R_i^2)^{\frac14}}
   ((1+\delta)\theta+\delta D_0)+c(i).
 \end{split} 
\end{equation*}
 Since $\delta$ and $\theta$ are arbitrary, 
 above estimate shows
 \begin{equation*}
  \limsup_{\beta\rightarrow -\infty}h(t+e^{\beta})=0\quad
   |\xi|\text{-almost all}\ (x,t)\in\overline\Omega\times(T_1,T_2)
 \end{equation*}
as well as \eqref{eq:6.12s}.  
This proves that $|\xi|((\overline\Omega\times(T_1,T_2))\setminus Z^-(T_1))=0$, since otherwise,
we have $\limsup_{\beta\rightarrow -\infty}h(t+e^{\beta})\geq \delta_0$ on a set of positive measure
with respect to $|\xi|$. Lemma \ref{lem:6.4} shows $\mu(Z^-(T_1))=0$, and since $|\xi|\leq \mu$ by
the definitions of these measures, we have $|\xi|(\overline\Omega\times(T_1,T_2))=0$.
\end{proof}

\section{Proof of main theorems}
In Section \ref{sec:5}, we have seen that there exists a subsequence such that $\mu_t^{\e_i}$ 
converges to $\mu_t$ for all $t\geq 0$. 
In this section we prove that the first variation of the limit varifold is bounded and rectifiable
for a.e$.$ $t\geq 0$. 
On the boundary $\partial \Omega$,
we show that the tangential component of the first variation is absolutely continuous
with respect to $\mu_t$ and prove at the end the desired limiting inequality
\eqref{mainineq}.

For each $u^{\e_i}$, we associate a varifold as follows.
\begin{definition}
For $\phi\in C(G_{n-1}(\overline\Omega))$, define
\begin{equation}
V_t^{\e_i}(\phi):=\int_{\Omega\cap \{|\nabla u^{\e_i}(t,\cdot)|\neq 0\} } \phi(x,I-a^{\e_i}\otimes
a^{\e_i})\, d\mu_t^{\e_i}(x).
\label{defvar}
\end{equation} 
Here, $a^{\e_i}
=\frac{\nabla u^{\e_i}}{|\nabla u^{\e_i}|}$.
\end{definition}
Note that we have $\|V_t^{\e_i}\|=\mu_t^{\e_i}\lfloor_{\{|\nabla u^{\e_i}
(t,\cdot)|\neq 0\}}$. 
We then derive a formula for the first variation of $V_t^{\e_i}$ up to the boundary. 
\begin{lemma} For $g\in C^1(\overline\Omega;{\mathbb R}^n)$, we have
\begin{equation}
\label{eq:rec0}
\begin{split}
\delta V_t^{\e_i}(g)&=\int_{\Omega} (g\cdot\nabla u^{\e_i})\big(\e_i\Delta u^{\e_i}
-\frac{W'}{\e_i}\big)\, dx+\int_{\Omega\cap \{|\nabla u^{\e_i}|\neq 0\}} \nabla g\cdot
(a^{\e_i}\otimes a^{\e_i})\xi^{\e_i}\, dx\\
& +\int_{\partial \Omega}(g\cdot \nu)\big(\frac{\e_i|\nabla u^{\e_i}|^2}{2}+\frac{W}{\e_i}
\big)-\int_{\Omega\cap \{|\nabla u^{\e_i}|=0\}} \nabla g\cdot I\, \frac{W}{\e_i}\, dx.
\end{split}
\end{equation}
\end{lemma}
\begin{proof}
Omit the sub-index $i$. 
We have 
\begin{equation}
\delta V_t^{\e} (g)=\int_{\Omega\cap\{|\nabla u^\e|\neq 0\}} \nabla g(x)\cdot (I
-a^{\e}\otimes a^{\e})\, d\mu_t^{\e}.
\label{eq:rec1}
\end{equation}
Using the boundary condition $\nabla u^\e \cdot \nu=0$ on $\partial\Omega$
and integration by parts, we have
\begin{equation}
\int_{\Omega} \nabla g\cdot I \, \frac{|\nabla u^{\e}|^2}{2}\, dx
=\int_{\partial\Omega}(g\cdot\nu)\frac{|\nabla u^{\e}|^2}{2}+\int_{\Omega}
\nabla g\cdot(\nabla u^{\e}\otimes\nabla u^{\e})+(g\cdot \nabla u^{\e})\Delta u^{\e}\, dx.
\label{eq:rec2}
\end{equation}
Also by integration by parts,
\begin{equation}
\int_{\Omega\cap \{|\nabla u^\e|\neq 0\}} W\nabla g\cdot I\, dx
=-\int_{\Omega\cap \{|\nabla u^\e|=0\}} W\nabla g\cdot I\, dx
-\int_{\Omega}(g\cdot \nabla u^\e)W'\, dx+\int_{\partial\Omega}
(g\cdot \nu)W.
\label{eq:rec3}
\end{equation}
Substituting \eqref{eq:rec2} and \eqref{eq:rec3} into \eqref{eq:rec1} and recalling 
the definition of $\xi^{\e}$, we obtain \eqref{eq:rec0}.
\end{proof}
\begin{prop}
\label{lemma7.2}
For a.e$.$ $t\geq 0$, $\mu_t$ is rectifiable on $\overline\Omega$, and any 
convergent subsequence $\{V_t^{\e_{i_j}}\}_{j=1}^{\infty}$ with 
\begin{equation}
\liminf_{j\rightarrow\infty}\Big\{\big( \int_{\Omega}\e_{i_j}\big(\Delta u^{\e_{i_j}}
-\frac{W'}{\e_{i_j}^2}\big)^2\, dx\big)^{\frac12}+\int_{\partial\Omega}\big(\frac{\e_{i_j}|\nabla u^{\e_{i_j}}|^2}{2}+\frac{W}{\e_{i_j}}\big)\Big\}
<\infty
\label{eq:rec4}
\end{equation}
(evaluated at $t$) converges to the unique varifold $V_t$ associated with $\mu_t$. Moreover we have
\begin{equation}
\|\delta V_t\|(\overline\Omega)<\infty
\label{eq:rec4.5}
\end{equation}
and 
\begin{equation}
\int_0^T\|\delta V_t\|(\overline\Omega)\, dt<\infty
\label{eq:rec4.6}
\end{equation}
for all $T<\infty$. 
\end{prop}
\begin{proof}
Due to the energy inequality, \eqref{surf} and Fatou's lemma, we have \eqref{eq:rec4} for
a.e$.$ $t\geq 0$ for the full sequence. Also for a.e$.$ $t\geq 0$, we have
\begin{equation}
 \lim_{i\rightarrow\infty}\int_{\Omega}|\xi^{\e_i}(t,\cdot)|\, dx=0
 \label{eq:rec5}
 \end{equation}
by Proposition \ref{prop:6.1} and the dominated convergence theorem. 
For such $t\geq 0$, there exists a converging subsequence $\{V_t^{\e_{i_j}}\}_{j=1}^{\infty}$ 
and a limit $V_t$ with \eqref{eq:rec4} satisfied. Then by \eqref{eq:rec0}, \eqref{eq:rec4} and
\eqref{eq:rec5}, we obtain for $g\in C^1(\overline\Omega;{\mathbb R}^n)$
\begin{equation}
\lim_{j\rightarrow\infty}|\delta V_t^{\e_{i_j}}(g)|\leq c(t)(\Cr{c-1}+1)\max_{\overline\Omega}|g|,
\label{eq:rec6}
\end{equation}
where we set $c(t)$ be the quantity \eqref{eq:rec4}. 
By the definition of varifold convergence, we have
\begin{equation}
|\delta V_t(g)|=\lim_{j\rightarrow\infty}|\delta V_t^{\e_{i_j}}(g)|\leq c(t)(\Cr{c-1}+1)\sup_{
\overline\Omega}|g|.
\label{eq:rec7}
\end{equation}
This shows that the total variation 
$\|\delta V_t\|$ is a Radon measure, showing \eqref{eq:rec4.5}. Since $\|V_t^{\e_{i_j}}\|=\mu_t^{\e_{i_j}}$,
we have $\|V_t\|=\mu_t$ which is uniquely determined. A covering argument
using the monotonicity formula (see the proof of \cite[Cor. 6.6]{arXiv:1307.6629}) shows 
\begin{equation}
{\mathcal H}^{n-1}({\rm spt}\, \mu_t)<\infty.
\label{eq:rec7.5}
\end{equation}
By \eqref{eq:rec7.5} (for more detail, see \cite[Prop. 6.11]{arXiv:1307.6629}) and \eqref{eq:rec7}, Allard's rectifiability theorem
shows that $V_t$ is a rectifiable varifold, and in particular, $V_t$ is determined uniquely by $\|V_t\|=\mu_t$.
This proves $\mu_t$ is rectifiable for a.e$.$ $t\geq 0$. The argument up to this point applies equally to 
any converging subsequence with \eqref{eq:rec4} and \eqref{eq:rec5}, thus the uniqueness of the limit varifold
follows. Since $c(t)$ is locally uniformly integrable, Fatou's lemma shows \eqref{eq:rec4.6}.
\end{proof}
We comment that the $\sigma^{-1}V_t\lfloor_{\Omega}\in {\bf IV}_{n-1}(\Omega)$ 
follows from the interior argument of \cite{Tonegawa1} or \cite{arXiv:1307.6629}.
Thus, up to this point, we proved Theorem \ref{theorem2.1} and \ref{theorem2.3}. 
We next prove Theorem \ref{theorem2.5}. 
\begin{prop}
For a.e$.$ $t\geq 0$ such that the claim of Proposition \ref{lemma7.2} holds, define 
$\delta V_t\lfloor_{\partial\Omega}^{\top}$ as in \eqref{deftan}. Then we have
\begin{equation}
\|\delta V_t\lfloor_{\partial\Omega}^{\top}+ \delta V_t\lfloor_{\Omega}\| \ll \|V_t\|
\label{eq:rec8}
\end{equation}
and writing the Radon-Nikodym derivative as
\begin{equation}
h_{b}(t) :=\left\{ 
\begin{array}{ll} - \frac{\delta V_t\lfloor_{\partial\Omega}^{\top}}{\|V_t\|} 
& \mbox{on }\partial \Omega,\\ 
- \frac{\delta V_t\lfloor_{\Omega}}{\|V_t\|}& \mbox{on } \Omega,
\end{array}\right.
\label{eq:rec8.5}
\end{equation}
we have \eqref{hel2} and
\begin{equation}
\int_{\overline\Omega} \phi |h_b|^2\, d\|V_t\|
\leq \liminf_{i\rightarrow\infty}\int_{\Omega} \e_i \big(\Delta u^{\e_i}-\frac{W'}{\e_i^2}\big)^2\phi\, dx
\label{eq:rec9}
\end{equation}
for $\phi\in C (\overline\Omega;{\mathbb R}^+)$. 
\label{lemma7.35}
\end{prop}
\begin{proof}
Let $V_t^{\e_{i_j}}$ be a subsequence converging to $V_t$. For any 
$g\in C^1_c(\Omega;{\mathbb R}^n)$, we may prove from \eqref{eq:rec0} that
\begin{equation}
|\delta V_t (g)|=\lim_{j\rightarrow\infty}|\delta V_t^{\e_{i_j}}(g)|\leq\big(\int_{\Omega} |g|^2\, d\|V_t\|\big)^{\frac12}
 \liminf_{j\rightarrow\infty}\big( \int_{\Omega} \e_{i_j}\big(\Delta u^{\e_{i_j}}-\frac{W'}{\e_{i_j}^2}
 \big)^2\, dx\big)^\frac12.
 \label{eq:rec9.5}
 \end{equation}
This shows that $\|\delta V_t\lfloor_{\Omega}\|\ll \|V_t\|$ and $\delta V_t\lfloor_{\Omega}
=-h(V_t,\cdot)\|V_t\|$ for $h(V_t,\cdot)\in L^2(\|V_t\|)$.
Next, given arbitrary $\epsilon>0$, let $\nu^{\epsilon}\in C^1(\overline\Omega;{\mathbb R}^n)$ be such
that $\nu^{\epsilon}\lfloor_{\partial\Omega}=\nu$, $|\nu^{\epsilon}|\leq 1$ and ${\rm spt}\,\nu^{\epsilon}
\subset N_{\epsilon}$. For $g\in C^1(\overline\Omega;{\mathbb R}^n)$, define $\tilde g:= g-(\nu^{\epsilon}
\cdot g)\nu^{\epsilon}$. Then, we have $\tilde g\cdot\nu=0$ on $\partial\Omega$ thus
$\delta V_t\lfloor_{\partial\Omega}^{\top}(g)=\delta V_t\lfloor_{\partial\Omega}^{\top}(\tilde g)$. To prove \eqref{eq:rec8},
we note
\begin{equation}
\begin{split}
\delta V_t\lfloor_{\partial\Omega}^{\top}(g)&+\delta V_t\lfloor_{\Omega}(g)=\delta V_t\lfloor_{\partial \Omega}
(\tilde g)+\delta V_t\lfloor_{\Omega}(\tilde g)+\delta V_t\lfloor_{\Omega}(g-\tilde g) \\
&=\delta V_t(\tilde g)+\delta V_t\lfloor_{\Omega}(g-\tilde g) \\
&=\lim_{j\rightarrow\infty} \delta V_t^{\e_{i_j}}(\tilde g)
+\delta V_t\lfloor_{\Omega}(g-\tilde g) \\
&= \lim_{j\rightarrow\infty}\int_{\Omega}(\tilde g\cdot \nabla u^{\e_{i_j}})
(\e_{i_j}\Delta u^{\e_{i_j}}-\frac{W'}{\e_{i_j}})\, dx+\delta V_t\lfloor_{\Omega}(g-\tilde g) \\
&\leq \big(\int_{\overline\Omega} |g|^2\, d\|V_t\|\big)^{\frac12}
\liminf_{j\rightarrow\infty} \big(\int_{\Omega} \e_{i_j}\big(\Delta u^{\e_{i_j}} -\frac{W'}{\e_{i_j}^2}\big)^2\, dx\big)^{\frac12}
+\delta V_t\lfloor_{\Omega}(g-\tilde g)
\end{split}
\label{eq:rec10}
\end{equation}
where we used \eqref{eq:rec0}, \eqref{eq:rec5} and $|\tilde g|\leq |g|$. 
Since ${\rm spt}\, \nu^{\epsilon}\subset N_{\epsilon}$, we have
\begin{equation}
|\delta V_t\lfloor_{\Omega}(g-\tilde g)|=\big|\int_{\Omega} h(V_t,\cdot)\cdot (g-\tilde g)
\, d\|V_t\|\big|
\leq \sup|g| \int_{N_{\epsilon}}|h(V_t,\cdot)|\, d\|V_t\|\rightarrow 0
\label{eq:rec11}
\end{equation}
as $\epsilon\rightarrow 0$. Then \eqref{eq:rec10} and \eqref{eq:rec11} show \eqref{eq:rec9}
with $\phi=1$. For general $\phi\in C(\overline\Omega;{\mathbb R}^+)$, we may 
carry out an approximation argument to obtain \eqref{eq:rec9} (see \cite[Prop. 8.2]{arXiv:1307.6629}
for the detail). 
\end{proof}
The inequality \eqref{hel2} follows from \eqref{eq:rec9} with $\phi=1$ and \eqref{engu2}. 
\begin{prop}
\label{lemma7.4}
Let $t\geq 0$ and $\{V_t^{\e_{i_j}}\}_{j=1}^{\infty}$ be as in Proposition \ref{lemma7.2},
and define $h_b(t)$ as in \eqref{eq:rec8.5}. Then we have
\begin{equation}
\lim_{j\rightarrow\infty}\int_{\Omega}(g\cdot\nabla u^{\e_{i_j}})\big(
\e_{i_j}\Delta  u^{\e_{i_j}}-\frac{W'}{\e_{i_j}}\big)\, dx=- \int_{\overline\Omega}
g\cdot h_b(t)\, d\|V_t\|
\label{eq:rec12}
\end{equation}
for $g\in C^1(\overline\Omega;{\mathbb R}^n)$ with $g\cdot\nu=0$ on $\partial
\Omega$.
\end{prop}
\begin{proof}
Since $V_t^{\e_{i_j}}$ converges to $V_t$ as varifold, we have 
$\lim_{j\rightarrow\infty}\delta V_t^{\e_{i_j}}(g)=\delta V_t(g)$. On the 
right-hand side of \eqref{eq:rec0}, since $g\cdot\nu=0$ on $\partial\Omega$, 
the third boundary integral term vanishes. Then we have \eqref{eq:rec12} 
from \eqref{eq:rec0}, \eqref{eq:rec5} and $\delta V_t(g)=\delta V_t\lfloor_{\Omega}(g)+
\delta V_t\lfloor_{\partial\Omega}^{\top}(g)$. 
\end{proof}

Finally we give
\begin{proof}
[Proof of Theorem \ref{theorem2.6}] It is enough to 
prove \eqref{mainineq} for $\phi\in C^2(\overline\Omega\times[0,\infty)\,;\,
{\mathbb R}^+)$ with $\nabla\phi(\cdot,t)\cdot \nu=0$ on $\partial \Omega$.
By writing $f^{\e_i}:=-\e_i\Delta u^{\e_i}+\frac{W'}{\e_i}$, we have
from \eqref{eq:1.1}
\begin{equation}
\int_{\Omega} \phi\, d\mu_t^{\e_i}\Big|_{t=t_1}^{t_2}
=\int_{t_1}^{t_2}\Big( \int_{\Omega}-\frac{1}{\e_i}(f^{\e_i})^2\phi
+f^{\e_i}\nabla\phi\cdot\nabla u^{\e_i}\, dx
+\int_{\Omega}\partial_t\phi\,d\mu_t^{\e_i}\Big)\, dt.
\label{eq:rec13}
\end{equation}
Since we already know that $\mu_t^{\e_i}\rightharpoonup \|V_t\|$
for all $t\geq 0$, the left-hand side of \eqref{eq:rec13} converges to 
that of \eqref{mainineq}, and so is the last term of the right-hand side. 
So we only need to consider the first and second terms of the right-hand 
side. Just as in the proof of Lemma \ref{lem:5.1}, $\int_{\Omega}
(\e_i^{-1}(f^{\e_i})^2\phi-f^{\e_i}\nabla\phi\cdot\nabla u^{\e_i})\, dx\geq 
-\Cr{c-1}\|\phi\|_{C^2}$. Thus by Fatou's lemma,
\begin{equation}
\lim_{i\rightarrow\infty} \int_{t_1}^{t_2}\int_{\Omega}\frac{1}{\e_i}(f^{\e_i})^2\phi
-f^{\e_i}\nabla\phi\cdot\nabla u^{\e_i}\, dxdt
 \geq \int_{t_1}^{t_2}\liminf_{i\rightarrow\infty}
\int_{\Omega}\frac{1}{\e_i}(f^{\e_i})^2\phi-f^{\e_i}\nabla\phi\cdot\nabla u^{\e_i}
\, dxdt.
\label{eq:rec14}
\end{equation}
Thus from \eqref{eq:rec13} and \eqref{eq:rec14}, we will finish the proof if we 
prove
\begin{equation}
\liminf_{i\rightarrow\infty}\int_{\Omega}\frac{1}{\e_i}(f^{\e_i})^2\phi
-f^{\e_i}\nabla\phi\cdot\nabla u^{\e_i}\, dx\geq \int_{\overline\Omega}
\phi|h_b|^2-h_b\cdot\nabla\phi\, d\|V_t\|
\label{eq:rec15}
\end{equation}
for a.e$.$ $t\in [t_1,t_2]$. For a.e$.$ $t$ where the assumption of Proposition \ref{lemma7.2}
is satisfied, we have already proved \eqref{eq:rec9} and \eqref{eq:rec12}. But this 
shows precisely \eqref{eq:rec15}. This ends the proof of \eqref{mainineq}.
\end{proof}
\section{Final remarks}
\label{finalremark}
It seems likely that, if $\|V_0\|(\partial\Omega)=0$, then $\|V_t\|(\partial\Omega)=0$
holds for all $t>0$. Intuitively, due to the strict convexity of the domain and the
Neumann boundary condition (which should intuitively imply 90 degree angle of intersection), 
interior of moving hypersurfaces should not touch $\partial \Omega$. 
Due to the maximum principle, this cannot
happen if the hypersurfaces are smooth up to the boundary. 
But within the general framework of this paper, we do not know
how to prove such statement or if it is indeed true. 

Though it may first appear counter intuitive in view of the connection to the MCF, 
if we have $\|V_0\|(\partial\Omega)>0$, then it is possible to have $\|V_t\|(\partial\Omega)>0$
for all $t>0$. An example can be provided by a limit of time-independent solutions
of \eqref{eq:1.1} where $\mu^{\e}\rightharpoonup c
{\mathcal H}^{n-1}\lfloor_{\partial\Omega}$ on $\overline\Omega$ as $\e\rightarrow 0$, where $c>0$
is some constant. One can obtain such family of solutions $u^{\e}$ by considering $\Omega=B_1$ 
and a mountain path solution connecting two constant functions $1$ and $-1$ 
within a class of radially symmetric functions. There are uniform positive lower and upper bounds of $E^{\e}(u^{\e})$
and the limiting varifold $V$ is non-trivial. On the other hand, if $\|V\|(B_1)>0$, 
due to \cite{Hutchinson}, ${\rm spt}\,\|V\|$ has to be a minimal surface, 
which contradicts the radially symmetry. Thus $\|V\|$
is concentrated only on $\partial B_1$ and is non-trivial. In this particular case, note that
$\delta V=-\frac{x}{|x|}{\mathcal H}^{n-1}\lfloor_{\partial B_1}$ 
and the tangential component $\delta V\lfloor_{\partial B_1}^{\top}$ is 0.
Using more explicit and sophisticated method, Malchiodi-Ni-Wei \cite{Malchiodi2} constructed a family of
solutions with multiple layers whose energy concentrates on $\partial B_1$ with $\|V\|(\partial B_1)=
N\sigma {\mathcal H}^{n-1}$, $N\in {\mathbb N}$. $N$ may be arbitrarily chosen. 
Furthermore, for general strictly mean convex domain $\Omega$, Malchiodi-Wei \cite{Malchiodi} 
constructed a family of single layered solutions whose limit energy concentrates on $\partial \Omega$. 
Even though such limit measures are not certainly the MCF in ${\mathbb R}^n$ 
in the usual sense (it should shrink),
such time independent 
measures satisfies \eqref{mainineq} trivially since $h_b=0$. This is the reason that
we need to decompose the first variation on $\partial\Omega$ to accommodate such 
cases in general. 

The existence result of the present paper suggests a reasonable setting for  proving
the boundary regularity of MCF. It is
interesting to extend interior regularity theorem (see \cite{MR0485012,arXiv:1111.0824,Tonegawa2})
to the corresponding boundary regularity theorem. For the time-independent case, 
interior regularity \cite{Allard} has been extended to boundary regularity
\cite{MR0397520,arXiv:1008.4728,MR0863638}.

It is worthwhile to comment on the strict convexity assumption. The places
the condition played any role are in the proof of Proposition \ref{prop:4.1} via Lemma \ref{lem:4.1}, 
and in some computations such as \eqref{eq:6s3} and before \eqref{eq:6.13}. Even without
strict convexity on the whole of $\Omega$, one can in fact localize these arguments. Namely,
for general bounded domain with smooth boundary $\Omega$, let $\Gamma\subset \partial\Omega$
be a set of points with some non-positive principal curvature. Then one can carry out the argument of this 
paper for $\overline\Omega\setminus (\Gamma)_{\epsilon}$, where $(\Gamma)_{\epsilon}$ is the 
$\epsilon$-neighborhood of $\Gamma$. All the statements in Section \ref{main} hold with $\overline\Omega
\setminus \Gamma$ in place of $\overline\Omega$. We did not write the paper in this generality 
to avoid further notational complications. 
\section{Appendix}
\label{sec:A}
We include a lemma which appeared in \cite{MR1237490} for reader's convenience.
\begin{equation*}
 \rho_y^r(x)
  :=\frac{1}{(\sqrt{2\pi}r)^{n-1}}\exp\left(-\frac{|x-y|^2}{2r^2}\right).
\end{equation*}
Then, $\rho_{(y,s)}=\rho_y^r$ when $r^2=2(s-t)$.

\begin{lemma}
 \label{lem:A.1}
 Let $\mu$ be a Radon measure on $\R^n$ satisfying for some $D>0$
 \begin{equation}
   \label{eq:A.2}
    \frac{\mu(B_R(x))}{\omega_{n-1}R^{n-1}}\leq D
  \end{equation}
 for $R>0$ and $x\in\R^n$. Then we obtain the following:
 \begin{enumerate}[(1)]
  \item For $r>0$ and for $x\in\R^n$,
	\begin{equation*}
	 \int_{\R^n}\rho_x^r\,d\mu\leq D
	\end{equation*}	
  \item For $r,R>0$ and for $x\in\R^n$,
	\begin{equation*}
	 \int_{\R^n\setminus B_R(x)}\rho_x^r\,d\mu
	  \leq 2^{n-1}e^{-\frac{3R^2}{8r^2}}D
	\end{equation*}
  \item For $\delta>0$ there is $\gamma_1>0$ depending only on $n$ and
	$\delta$ such that for $x,x_0\in\R^n$ and $r>0$ satisfying
	$|x-x_0|<\gamma_1r$ we have
	\begin{equation*}
	 \int_{\R^n}\rho_{x_0}^r\,d\mu
	  \leq (1+\delta)\int_{\R^n}\rho_x^r\,d\mu
	  +\delta D.
	\end{equation*}
  \item For $\delta>0$ there is $\gamma_2>0$ depending only on $n$ and
	$\delta$ such that for $x\in\R^n$ and $r,R>0$ satisfying
	$1\leq\frac{R}{r}\leq1+\gamma_2$ we have
	\begin{equation*}
	 \int_{\R^n}\rho_{x}^R\,d\mu
	\leq (1+\delta)\int_{\R^n}\rho_x^r\,d\mu
	+\delta D.
	\end{equation*}
 \end{enumerate}
\end{lemma}




\begin{thebibliography}{99}
 \bibitem{Allard}
	 Allard, W. K.,
	 \emph{On the first variation of a varifold},
	 Ann. of Math. (2) \textbf{95} (1972), 
	 417--491.

 \bibitem{MR0397520}
	 Allard, W. K.,
	 \emph{On the first variation of a varifold: boundary
	 behavior}, 
	 Ann. of Math. (2) \textbf{101} (1975), 
	 418--446.
	 
 \bibitem{arXiv:1008.4728}
	 Bourni, T., 
	 \emph{Allard type boundary regularity theorem for varifolds
	 with $C^{1,\alpha}$ boundary}, 
	 arXiv:1008.4728.
	 
 \bibitem{MR0485012} 
	 Brakke, K.~A., 
	 \emph{The motion of a surface by its mean curvature},
	 Mathematical Notes, vol.~20, Princeton University Press, 
	 1978.

 \bibitem{MR1101239}
	 Bronsard, L. and Kohn, R. V.,  
	 \emph{Motion by mean curvature as the singular limit of
	 {G}inzburg-{L}andau dynamics},
	 J. Differential Equations \textbf{90} (1991),
	 211--237.

 \bibitem{MR2180601}
	 Buckland, J. A.,
	 \emph{Mean curvature flow with free boundary on smooth
	 hypersurfaces},
	 J. Reine Angew. Math. \textbf{586} (2005),
	 71--90.
	 
 \bibitem{MR0480282}
	 Casten, R.\,G. and Holland, C.\,J.,
	 \emph{Instability results for reaction diffusion equations with
	 {N}eumann boundary conditions},
	 J. Differential Equations \textbf{27} (1978),
	 266--273.
	 
 \bibitem{MR1153311}
	 Chen, X.,
	 \emph{Generation and propagation of interfaces for
	 reaction-diffusion equations},
	 J. Differential Equations \textbf{96} (1992),
	 116--141.
	 
 \bibitem{MR1055457}
	 de Mottoni, P. and Schatzman, M.,
	 \emph{\'{E}volution g\'eom\'etrique d'interfaces},
	 C. R. Acad. Sci. Paris S\'er. I Math. \textbf{309} (1989),
	 453--458.

 \bibitem{MR1177477}
	 Evans, L. C., Soner, H. M. and Souganidis, P. E.,
	 \emph{Phase transitions and generalized motion by mean curvature},
	 Comm. Pure Appl. Math. \textbf{45} (1992),
	 1097--1123.

 \bibitem{MR1235189}
	 Giga, Y. and Sato, M.-H.,
	 \emph{Neumann problem for singular degenerate parabolic equations}, 
	 Differential Integral Equations \textbf{6} (1993),
	 1217--1230.

\bibitem{Gilbarg}
      	Gilbarg, D. and Trudinger, N. S.,
	\emph{Elliptic partial differential equations of second order}, 2nd ed. ,
	Springer, 1983.
	
 \bibitem{MR0863638}
	 Gr\"uter, M. and Jost, J., 
	 \emph{Allard type regularity results for varifolds with free
	 boundaries}, 
	 Ann. Scuola Norm. Sup. Pisa Cl. Sci. (4) \textbf{13} (1986), 
	 129--169.

 \bibitem{MR1030675}
	 Huisken, G.,
	 \emph{Asymptotic behavior for singularities of the mean
	 curvature flow}, 
	 J. Differential Geom. \textbf{31} (1990),
	 285--299.
	 
 \bibitem{Hutchinson}
	 Hutchinson, J. E. and Tonegawa, Y.,
	 \emph{Convergence of phase interfaces in the van der Waals-Cahn-Hilliard theory},
	 Calc. Var. Partial Differential Equations \textbf{10} (2000), 
	 49--84.
     
 \bibitem{MR1237490}
	 Ilmanen, T., 
	 \emph{Convergence of the {A}llen-{C}ahn equation to {B}rakke's
	 motion by mean curvature}, 
	 J. Differential Geom. \textbf{38} (1993), 
	 417--461.

 \bibitem{arXiv:1111.0824} 
	 Kasai, K. and Tonegawa, Y., 
	 \emph{A general regularity theory for weak mean curvature
	 flow},
	 Calc. Var. Partial Differential Equations \textbf{50} (2014),
	 1--68.

 \bibitem{MR1341031}
	 Katsoulakis, M., Kossioris, G. T. and Reitich, F.,
	 \emph{Generalized motion by mean curvature with Neumann
	 conditions and the Allen-Cahn model for phase transitions},
	 J. Geom. Anal. \textbf{5} (1995),
	 255--279.


 \bibitem{MR2886118}
	 Koeller, A. N., 
	 \emph{Regularity of mean curvature flows with {N}eumann free
	 boundary conditions},
	 Calc. Var. Partial Differential Equations \textbf{43} (2012), 
	 265--309.


 \bibitem{MR0241822}
	 Lady{\v{z}}enskaja, O. A., Solonnikov, V. A. and
	 Ural'ceva, N. N.,
	 \emph{Linear and quasilinear equations of parabolic type},
	 American Mathematical Society,
	 Providence, R.I.,
	 1967.
	 
 \bibitem{MR2652019}
	 Liu, C., Sato, N. and Tonegawa, Y.,
	 \emph{On the existence of mean curvature flow with transport term},
	 Interfaces Free Bound. \textbf{12} (2010),
	 251--277.


 \bibitem{Malchiodi}
	 Malchiodi, A. and Wei, J.,
	 \emph{Boundary interface for the Allen-Cahn equation}, 
	 J. Fixed Point Theory Appl. \textbf{1} (2007), 
	 305--336.
 
 \bibitem{Malchiodi2}
	 Malchiodi, A., Ni, W.-M. and Wei, J.,
	 \emph{Boundary-clustered interfaces for the Allen-Cahn equation},
	 Pacific J. Math. \textbf{229} (2007), 
	 447--468. 

 \bibitem{MR0555661}
	 Matano, H.,
	 \emph{Asymptotic behavior and stability of solutions of semilinear
	 diffusion equations},
	 Publ. Res. Inst. Math. Sci. \textbf{15}, (1979),
	 401--454.
	
 \bibitem{Modica}
	 Modica, L.,
	 \emph{Gradient theory of phase transitions and minimal interface criteria},
	 Arch. Rational Mech. Anal. \textbf{98} (1987),
	 123--142.

 \bibitem{MR2383536}
	 Mugnai, L. and R{\"o}ger, M.,
	 \emph{The {A}llen-{C}ahn action functional in higher dimensions},
	 Interfaces Free Bound. \textbf{10} (2008),
	 45--78.

 \bibitem{MR978829}
	 Rubinstein, J., Sternberg, P. and Keller, J. B.,
	 \emph{Fast reaction, slow diffusion, and curve shortening},
	 SIAM J. Appl. Math.
	 \textbf{49} (1989),
	 116--133.

 \bibitem{MR1287918}
	 Sato, M.-H.,
	 \emph{Interface evolution with Neumann boundary condition}, 
	 Adv. Math. Sci. Appl. \textbf{4} (1994),
	 249--264. 
	 
 \bibitem{MR2440879}
	 Sato, N.,
	 \emph{A simple proof of convergence of the {A}llen-{C}ahn equation
	 to {B}rakke's motion by mean curvature},
	 Indiana Univ. Math. J. \textbf{57} (2008),
	 1743--1751.

 \bibitem{MR756417}
	 Simon, L.,
	 \emph{Lectures on geometric measure theory},
	 Proceedings of the Centre for Mathematical Analysis,
	 Australian National University,
	 \textbf{3},
	 Australian National University Centre for Mathematical
	 Analysis, Canberra, 1983.

 \bibitem{Soner1}
	 Soner, H. M.,
	 \emph{Convergence of the phase-field equations to the
	 {M}ullins-{S}ekerka problem with kinetic undercooling},
	 Arch. Rational Mech. Anal. \textbf{131} (1995),
	 139--197.

 \bibitem{MR1674799}
	 Soner, H. M.,
	 \emph{Ginzburg-{L}andau equation and motion by mean curvature. {I}.
	 {C}onvergence},
	 J. Geom. Anal. \textbf{7} (1997),
	 437--475.

 \bibitem{Soner3}
	 Soner, H. M.,
	 \emph{Ginzburg-{L}andau equation and motion by mean curvature. {II}.
	 {D}evelopment of the initial interface},
	 J. Geom. Anal. \textbf{7} (1997),
	 477--491.

\bibitem{MR1393271}
	Stahl, A., 
	\emph{Regularity estimates for solutions to the mean curvature flow
	with a {N}eumann boundary condition},
	Calc. Var. Partial Differential Equations \textbf{4} (1996), 
	385--407.

\bibitem{MR1402731}
	Stahl, A.,
	\emph{Convergence of solutions to the mean curvature flow with a
	{N}eumann boundary condition},
	Calc. Var. Partial Differential Equations \textbf{4} (1996), 
	421--441.

\bibitem{MR1620498}
	Sternberg, P. and Zumbrun, K.,
	\emph{Connectivity of phase boundaries in strictly convex
	domains},
	Arch. Rational Mech. Anal. \textbf{141} (1998),
	375--400.
	
 \bibitem{arXiv:1307.6629}
	 Takasao, K. and Tonegawa, Y.,
	 \emph{Existence and regularity of mean curvature flow with
	 transport term in higher dimensions},
	 arXiv:1307.6629v2.
	 
 \bibitem{MR1970021}
	 Tonegawa, Y., 
	 \emph{Domain dependent monotonicity formula for a singular
	 perturbation problem}, 
	 Indiana Univ. Math. J. \textbf{52} (2003),
	 69--83.
	
 \bibitem{Tonegawa1}
	 Tonegawa, Y., 
	 \emph{Integrality of varifolds in the singular limit of
	 reaction-diffusion equations}, 
	 Hiroshima Math. J. \textbf{33} (2003), 
	 323--341.
     
 \bibitem{Tonegawa2}
	 Tonegawa, Y., 
	 \emph{A second derivative H\"{o}lder estimate for weak mean
	 curvature flow}, 
	 Adv. Calc. Var. \textbf{7} (2014), 
	 91--138.
\end{thebibliography}
\end{document}